\documentclass[11pt]{article}
\usepackage{amsmath, amsfonts, amssymb, subfigure, amsthm, graphicx, color, ulem, verbatim, epstopdf,enumerate,array,multirow}
\usepackage{hyperref,eq_queue_defs}

\usepackage{fullpage}

\newtheorem{cor}{Corollary}[section]

\begin{document}
\title{Batch latency analysis and phase transitions for a tandem of queues with exponentially distributed service times}

\author{Jinho Baik\footnote{Department of Mathematics, University of Michigan, Ann Arbor, MI 48109, USA} and~Raj Rao Nadakuditi\footnote{Department of Electrical and Computer Engineering, University of Michigan, Ann Arbor,
MI, 48109, USA}}

\maketitle

\begin{abstract}
We analyze the latency or sojourn time $L(m,n)$ for the last customer in a batch of $n$ customers to exit from the $m$-th queue in a tandem of $m$ queues  {in the setting where  the queues are in equilibrium before the batch of customers arrives at the first queue.
We first characterize the distribution of $L(m,n)$ exactly for every $m$ and $n$, under the assumption that the queues have unlimited buffers and that each server has customer independent, exponentially distributed service times with an arbitrary, known rate.
We then evaluate the first two leading order terms of the distributions in the  large $m$ and $n$ limit and bring into sharp focus the existence of phase transitions in the system behavior.
The phase transition occurs due to the presence of either slow bottleneck servers or a high external arrival rate.
We determine the critical thresholds for the service rate and the arrival rate, respectively,  about which this phase transition occurs; it turns out that they are the same.}
This critical threshold depends, in a manner we make explicit, on the individual service rates, the number of customers and the number of queues but not on the external arrival rate.

\end{abstract}

\section{Introduction}

Tandem queues are important models for production systems \cite{li2008production} and communication networks \cite{xie2009towards}. The fact that the output process of a server is the arrival process for the subsequent queue makes tandem queues difficult to analyze. There are several results in the literature on the waiting time \cite{reich1957waiting} of customers in such queues \cite{friedman1965reduction,Tembe-Wolff74}, moment generating functions \cite{fidler2006end}, scaling properties \cite{ciucu2006scaling}, heavy traffic approximations \cite{harrison1978diffusion} and bounds thereof \cite{van1988simple,niu1980bounds,niu1981comparison}, to list a few.  These results capture the behavior of a single randomly selected customer. 
In this paper, in contrast, we are interested in the latency or sojourn time of a \textit{batch} of $n$ customers entering a system of $m$ tandem queues with unlimited buffers that are in equilibrium 
due to an external arrival process of rate $\alpha$ at the first queue before the first customer of the batch arrives; the same setup with $\alpha=0$ was considered by Glynn and Whitt in their seminal paper \cite{Glynn-Whitt91} and more recently by Baccelli, Borovkov and Mairesse  in \cite{Baccelli99asymptoticresults}.  
\textcolor{black}{The batch latency  analysis problem is motivated by networking applications such as peer-to-peer sharing where a (large) file is transferred from a source to a destination over a large multi-hop network \cite{conti2007multihop,ding2004peer}.}

This setting is considerably more difficult to analyze at a level that captures what happens as the  ``wave'' of customers traverses the network.  Of particular relevance to this work is using the analysis to provide insights on what happens when a few bottleneck servers are particularly slow or when the rate $\alpha$ of the external arrival process is relatively large.

%\textcolor{red}{The latency time in a tandem queue is well known to be recast as the last passage time in a directed last passage percolation (DLPP) \cite{Tembe-Wolff74,Glynn-Whitt91}. We show that }

In this paper, we leverage recent results from directed last passage percolation (DLPP) and random matrix theory %(e.g. \cite{Borodin-Peche08a})
(e.g. \cite{Johansson00,Okounkov01a,Borodin-Peche08a})
to show that the latency distribution equals that of the largest eigenvalue of a specially constructed random matrix (see Theorem \ref{thm:exact}).

\begin{comment}
The connection between batch latency in a tandem of queues, and directed last passage percolation and random matrix theory
was also discussed and used, for example, in \cite{OConnell00,OConnell03,Draief-Mairesse-OConnell,Martin}.
An immediate goal in deriving the distribution is to document in this paper a very  general result attainable using the underlying theory.
\end{comment}

{
We then employ $m,n \to \infty$ asymptotics to uncover a phase transition which separates a regime in which the presence of slow bottleneck servers results in a latency that differs 
from the case where there are no slow servers. 
It is shown that %, assuming there are finitely many slow servers, they do not affect the latency to the leading order unless their rates are smaller than a critical rate.
the leading order of the latency is affected only if the slowest server has the service rate below a critical threshold.
We also show that the next order fluctuations change in this case with different order of the variance.
%In addition to the leading order asymptotics, we also consider the next order fluctuations and show that the fluctuations of the latency are normally distributed with variance $O(n)$ when sufficiently slow servers are present but the fluctuations are given by the complex Tracy-Widom distribution with variance $O(n^{2/3})$ when the rates of the slow servers are above the critical value.
The analysis highlights why optimizing the service rates of a limited or $o(n)$  number of bottleneck servers might produce asymptotically vanishing gains in the latency reduction achieved.
}

{
A similar phase transition occurs with respect to the external arrival rate $\alpha$. 
In this case the leading order of the latency is affected if $\alpha$ is higher than a critical threshold.
It is shown that this critical threshold is the same as the one that arises from the slow server phase transition analysis. 
}

%Our emphasis on uncovering the phase transitions in the presence of slow servers or when the arrival rate is what differentiates this work from related work in the DLPP literature. 

\begin{comment}
It is well-known that the batch latency for a tandem of queues is related to
directed last passage percolation (DLPP) models \cite{Glynn-Whitt91}.
Among the DLPP models, there are certain `solvable' models which are known to be connected to random matrix theory (RMT) \cite{Johansson00,Borodin-Peche08a}.
The analysis of these solvable DLPP models has been a very active area of research in probability and statistical physics (see e.g. \cite{Johansson02,Ferrari-Spohn,Corwin12}).
Combining these facts we arrive at the above result. See Section  \ref{sec:proofexact}. 

The above theorem allows one to use the method from RMT to study the batch latency.
Indeed RMT is a subject of rich history and techniques, and has been studied extensively in mathematics, statistics, physics, and engineering
\cite{Mehta04,Akemann-Baik-DiFrancesco11,Bai-Silverstein09}.
However, for the asymptotic results in the next subsections, we do not use this connection directly.
Instead we use the connection to solvable DLPP models mentioned above.
The key fact is that there are explicit formulas for the distribution functions of the last passage times for solvable DLPP models.
The asymptotic results in the next sections are obtained by analyzing these explicit formulas.
\end{comment}

{
%The analytical characterization of the latency distribution and the phase transition phenomenon are the main contributions of this paper, which is organized as follows.
The paper is organized as follows.
In Section \ref{sec:main results}, we state the problem addressed in this paper formally and present the main results. 
Some examples where our results may be applied along with a numerical computation are discussed in Section \ref{sec:examples}. 
The proofs of the main results are provided in Sections \ref{sec:proofexact} and \ref{sec:proof2}. Section \ref{sec:intuition} contains an intuitive explanation for the origin of the phase transition. 
}

\section{Main results}\label{sec:main results}

\subsection{Problem}

Consider a tandem of $m$ queues associated with $m$ servers labeled from left to right as $S_{1}, \ldots, S_{m}$.
We assume that each queue is $M/M/1$ and that the service time at server $i$ is exponentially distributed with a customer-independent rate equal to  $\mu_{i}$.
Each customer starts at queue $1$. After being served by $S_1$, the customer joins queue $2$ and waits for the turn to be served by $S_2$. After being served, the customer then joins  queue $3$ and so on until the customer exists  queue $m$.
We assume that the queue buffers are infinitely long and there are no external arrivals in the system except to queue $1$.
The arrival process to the first queue is an independent Poisson process of rate $\arr\ge 0$.
Suppose that the system is in equilibrium.
We assume that $\alpha < \mu_{i}$ for all $i$ so that the queues are stable.

Now suppose that we send in a batch of $n$ customers to queue $1$.
We assume that this arrival is according to the arrival Poisson process: we may think that when a new customer arrives at queue $1$ according to the Poisson process, we send in $n-1$ more customers immediately.
Let $L(m,n)$ denote the latency or sojourn time for the last of the batch of $n$ customers to exit from the last queue $m$.
Here the time is measured from the instance when the $n$ customers arrive at queue $1$.
If we assume that the batch of customers is sent in at an arbitrary time, rather than as a part of the arrival Poisson process, it is easy to see that the asymptotic results in this section still hold even though the exact result, Theorem~\ref{thm:exact}, does not hold.

{We note that the $\alpha = 0$ case corresponds to the setting where the queues are initially empty. The batch of $n$ customers are sent into  queue $1$ and the latency is measured from this point.}

\begin{comment}
In this paper we characterize the distribution of $L(m,n)$ for this system in both the finite and large $n$ and $m$ settings.
In the large $n$ and $m$ setting, we observe a phase transition phenomenon that the exit time does not change to the leading (and the next leading) order even though a few servers have slow rates than others if the rates are above a certain threshold.
There is also a similar threshold for the arrive rate $\alpha$ at which the leading order of the exit time changes.
\end{comment}

\subsection{Exact distribution}

{
The standard complex normal distribution is denoted by $\mathcal{CN}(0,1)$.
The notation $x \sim \mathcal{CN}(0,1)$ means that $ x= u + i \,  v$ where $u$ and $v$ are i.i.d. $\mathcal{N}(0,1/2)$}.

\begin{thm}[Exact distribution]
\label{thm:exact}
Define the diagonal matrices
$$\Sigma = \textrm{diag}(1/\mu_{1},\ldots,1/\mu_{m}) \textrm{ and } \Gamma = \textrm{diag}\left(1/(\mu_{1} - \alpha),\ldots,1/(\mu_{m}-\alpha)\right).$$
Let $G$ and $g$ denote an $m \times (n-1)$ matrix and an $m\times 1$ vector of i.i.d. $\mathcal{CN}(0,1)$ entries, respectively.
Consider the random Hermitian matrix
$$W = \Gamma^{1/2} gg^{*} \Gamma^{1/2} + \Sigma^{1/2}GG^{*} \Sigma^{1/2}.$$
Then we have that for all $n$ and $m$,
$$L(m,n) \overset{\mathcal{D}}{=} \Lambda_{\max}(W).$$
Here $\Lambda_{\max}(W)$ denotes the largest eigenvalue of $W$ and $\overset{\mathcal{D}}{=}$ denotes equality in distribution.
\end{thm}

{
This result follows from
interpreting the tandem queues model in terms of directed last passage percolation (DLPP) model \cite{Tembe-Wolff74,Glynn-Whitt91,Prahofer-Spohn02a} and then using a result due to Borodin and P\'ech\'e \cite{Borodin-Peche08a} on the relation between certain DLPP models and random matrices. See Section~\ref{sec:proofexact}.
}

\begin{rmk}\label{rmk:ordering}
Note that since the normal distribution is rotationally invariant, the distribution of the eigenvalues of $W$ is unchanged even if we re-arrange the ordering of $\mu_i$'s. Hence we find that the distribution of $L(m,n)$ does not depend on the ordering of $\mu_i$'s, a fact first proved in  \cite{Weber79}.
\end{rmk}

%Indeed the connection of $L(m,n)$ and $\lambda_{\max}(W)$ was obtained by noting that
%the distribution functions of both of them are precisely the same after having computed them separately.
%This will be discussed in Section  \ref{sec:solvable}.

\subsection{Asymptotic result I. Leading order}

%The next theorem is about the asymptotics of $L(m,n)$  in the large $n$, $m$ regime.
We first  introduce some definitions and assumptions.
Let
\begin{equation}
	\ser_{(m)}\le  \cdots\le \ser_{(2)} \le  \ser_{(1)}
\end{equation}
be an ordered re-arrangement of the services rates $\ser_1, \cdots, \ser_m$.
Consider the probability measure (the spectral measure for the service rates)
\begin{equation}
\begin{split}
	H_{m} =  \frac1{m} \sum_{i=1}^{m} \delta_{\ser_i}
	=  \frac1{m} \sum_{i=1}^{m} \delta_{\ser_{(i)}}.
\end{split}
\end{equation}
Assume that there is a compactly supported probability measure $H$ such that
\begin{equation}\label{eq:HmtoH}
\begin{split}
	H_{m} \to H \qquad \text{weakly}
\end{split}
\end{equation}
as $m \to \infty$ (i.e. $\int_{\R} g(x) dH_m(x) \to \int_{\R} g(x) dH(x)$ for all bounded continuous functions $g$.)
This includes  examples such as (a) $\mu_i=1$ for all $i$, (b) $\mu_{i_0}=\mu<1$ and $\mu_i=1$ for all $i\neq i_0$,
and (c) $\mu_i= F^{-1}(\frac{i}{m})$ for $F(x):= \int_{-\infty}^xf(y) dy$ for a piecewise-continuous, non-negative and compactly support function $f(y)$ with total mass $1$.
We denote the support of $H_m$ and $H$ by $supp(\meas_m)$ and $supp(\meas)$ respectively.
Note that the minimum of  $supp(\meas_m)$ is $\mu_{(m)}$.

Define the function
%\textcolor{red}{are the $m$ and $n$ here correct?}
\begin{equation}\label{eq:lz}
	l_m(z) := m \int \frac{dH_m(y)}{y-z} + \frac{n}{z}, \qquad z\notin supp(H_m).
\end{equation}
See Figure~\ref{fig:functionlfirst} for an example.
It is easy to see that $l_m''(z)>0$ for $z\in (0, \mu_{(m)})$.
{Since $l_m'(z)\to -\infty$ as $z\downarrow 0$ and $l_m'(z)\to +\infty$ as $z\uparrow \mu_{(m)}$,  there is a unique point $z\in (0, \mu_{(m)})$ such that $l_m'(z)=0$. }
We denote this point by $\lambda_m$;
\begin{equation}\label{eq:solvecrit}
	m\int \frac{dH_m(y)}{(y-\crit_m)^2} - \frac{n}{\crit_m^2} = 0 , \qquad 0< \crit_m< \mu_{(m)}.
\end{equation}
Note that $l_m'(z) <0$ for $0<z<\crit_m$ and $l_m'(z) >0$ for $\crit_m<z<\mu_{(m)}$.

\begin{figure}[htbp]
\centering
\subfigure[]
{  \includegraphics[scale=0.25]{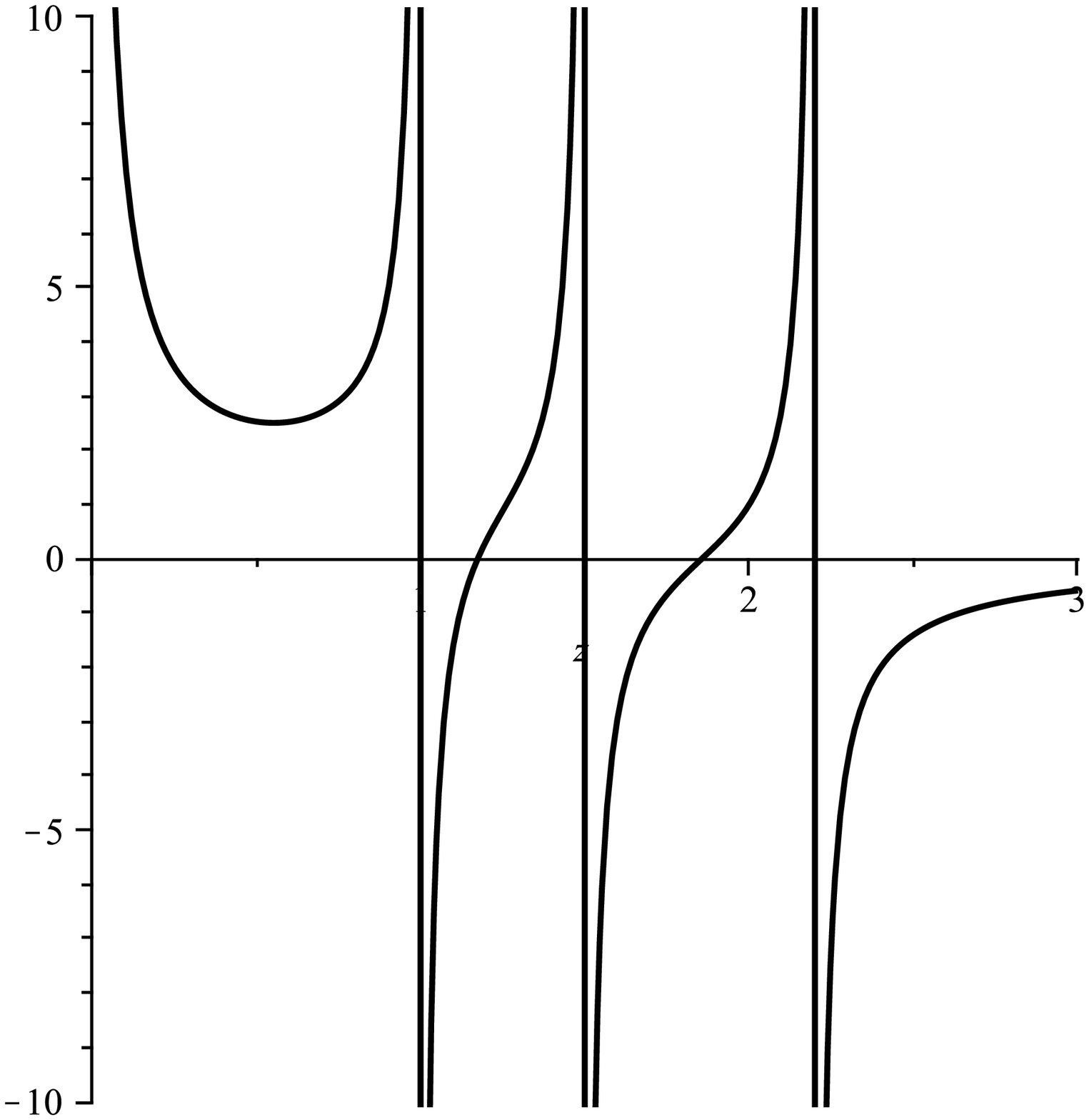}
}
\subfigure[]
{
 \includegraphics[scale=0.25]{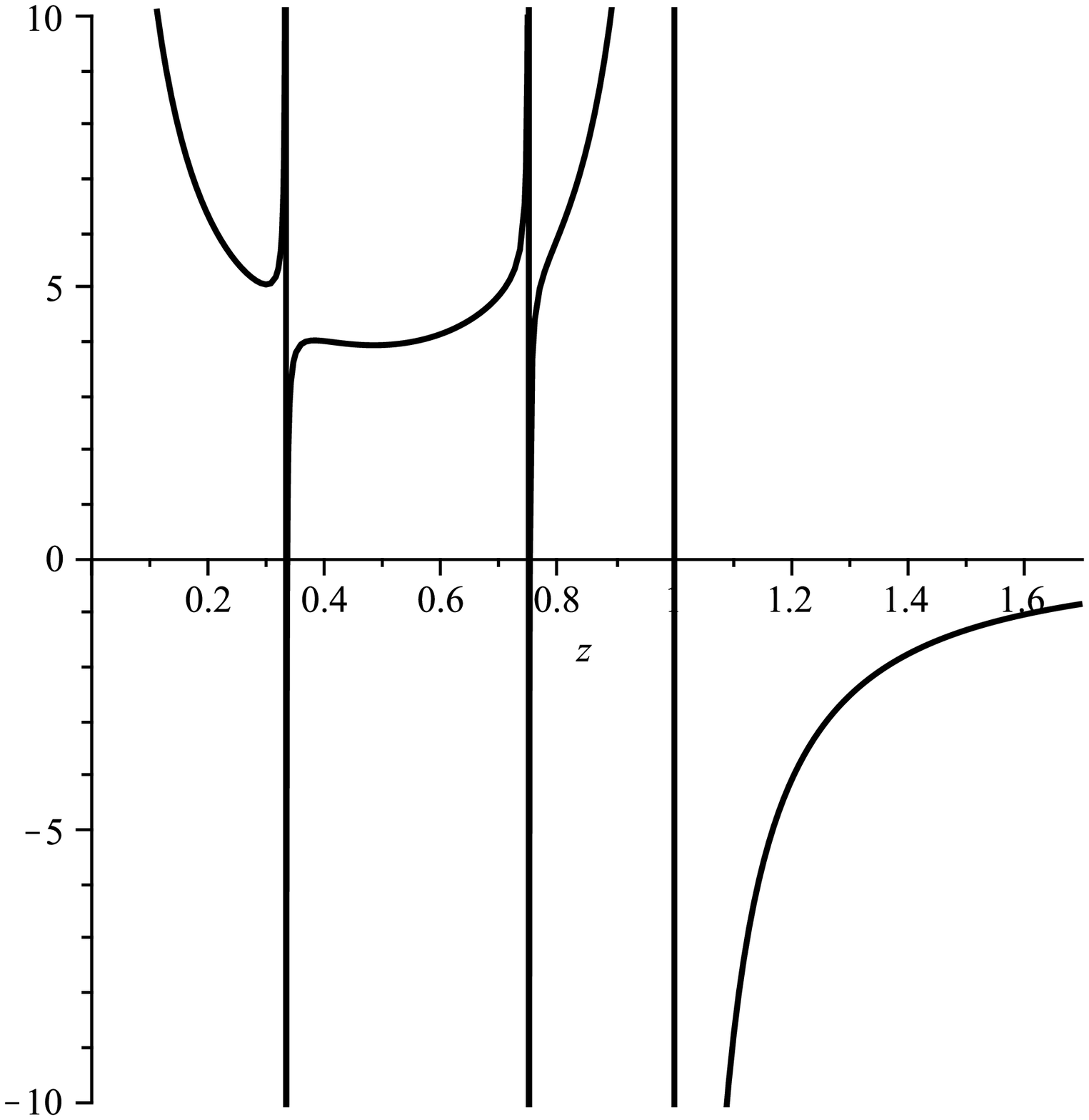}
}
\caption{Graphs of $\frac1{m}l_m(z)$. (a) is when $m=3, n=2$ and $\mu_1=1, \mu_2=1.5, \mu_3=2.2$.
(b) is when $m=100$, $n=100$, $\mu_1=\frac13$, $\mu_2=\frac34$, and $\mu_i=1$ for $i=3,\cdots, 100$.
In both cases, $\lambda_m$ is the argmin of $l_m(z)$ in the interval $(0, \mu_1)$.}
\label{fig:functionlfirst}
\end{figure}

%We are now ready to state the first asymptotic result.

\begin{thm}[Asymptotic result I. Leading order]\label{th:main}
Assume~\eqref{eq:HmtoH}.
Let $l_m(z)$ and $\crit_m$  be defined as in~\eqref{eq:lz}) and~\eqref{eq:solvecrit}) respectively.
Recall that $\alpha$ is the arrival rate and assume that $\alpha< \liminf_m \ser_{(m)}$.
From the definition, $\lambda_m< \ser_{(m)}$.
The following asymptotic result holds in probability
as $m,n\to\infty$ such that $m/n\to \gamma \in (0,\infty)$. % limit we have the following result.
\begin{enumerate}[(a)]
\item If $\alpha< \displaystyle{\liminf_{m\to\infty}} \crit_m$ and $\displaystyle{\liminf_{m\to\infty}} (\ser_{(m)}-\lambda_m)>0$, then
\begin{equation}
\begin{split}
	\frac{L(m,n)}{l_m(\crit_m)} \to 1.
\end{split}
\end{equation}

\item If $\displaystyle{\limsup_{m\to\infty}} \crit_m<\alpha$, %(note that necessarily $\alpha< \displaystyle{\liminf_{m\to\infty}}  \ser_{(m)}$ by assumption),
then
\begin{equation}
\begin{split}
	\frac{L(m,n)}{l_m(\alpha)}\to 1.
\end{split}
\end{equation}

\item Suppose that $\displaystyle{\limsup_{m\to\infty}}  (\ser_{(m)}-\lambda_m)=0$
and there is a fixed $r$ (which does not grow in $m$) such that
$\ser_{(m)}=\cdots= \ser_{(m-r+1)}$ and $\displaystyle{\liminf_{m\to\infty}}  (\ser_{(m-r)}-\ser_{(m-r+1)})>0$.
Furthermore, assume that $\displaystyle{\liminf_{m\to\infty}}  (\lambda^{(r)}_m-\mu_{(m)}) >0$ where
$\lambda^{(r)}_m$ is the unique real root of the equation
$(l^{(r)}_m)'(z)=0$ in $z\in (0, \mu_{(m-r)})$ where
\begin{equation}\label{eq:lmrde}
\begin{split}
	l_m^{(r)}(z):= m\int \frac{dH_m^{(r)}(y)}{y-z}+ \frac{n}{z},
	\qquad H_m^{(r)}:= \frac1{m} \sum_{i=1}^{m-r} \delta_{\mu_{(i)}}.
\end{split}
\end{equation}
Then
\begin{equation}
\begin{split}
	\frac{L(m,n)}{l_m^{(r)}(\ser_{(m)})}\to 1.
\end{split}
\end{equation}
\end{enumerate}
\end{thm}

The above theorem can be simplified if we assume that the limit of $\mu_{(m)}$ as $m\to\infty$ exists.

\begin{cor}[Easier conditions when $\mu_{(m)}\to \mu_{\inf}$]
% as $m\to \infty$ the slowest service rate converges to a constant.]
\label{rem:con}
Suppose that $\mu_{(m)}\to \mu_{\inf}$ as $m\to \infty$ for some  $\mu_{\inf}>\alpha$.
% (which is necessarily larger than $\alpha$).
%(Recall that $\mu_{(1)}= \mu_{(1)}(m)$ depends on $m$.)
Set
\begin{equation}\label{eq:lz 2}
	l(z) := m \int \frac{dH(y)}{y-z} + \frac{n}{z}, \qquad z\notin supp(H),
\end{equation}
and let $\crit$ be the unique solution to $l'(z)=0$ in $z\in (0, \inf supp H)$;% be the unique solution to the equation
\begin{equation}\label{eq:solvecrit2}
	m\int \frac{dH(y)}{(y-\crit)^2} - \frac{n}{\crit^2} = 0 , \qquad \crit \in (0,  \inf supp(H)).
\end{equation}
(Hence $l(z)$ is the analogue of $l_m(z)$ with $H_m$ replaced by $H$ in ~\eqref{eq:HmtoH},
and $\lambda$ is the analogue of $\lambda_m$. See Figure~\ref{fig:functionlwithH} for an example.)
Then the following asymptotic result holds in probability
as $m,n\to\infty$ such that $m/n\to \gamma \in (0,\infty)$.
%Then we have:
\begin{itemize}
\item[(i)] If $\lambda\in (\alpha, \mu_{\inf})$, then $\frac{L(m,n)}{l(\lambda)}\to 1$.
\item[(ii)] If $\lambda\in (0, \alpha)$, then $\frac{L(m,n)}{l(\alpha)}\to 1$.
\item[(iii)] If $\lambda>\mu_{\inf}$, and if $\mu_{(m)}=\cdots=\mu_{(m-r+1)}$ for some fixed $r$ and $\displaystyle{\liminf_{m\to\infty}} (\mu_{(m-r)}-\mu_{\inf})>0$, then $\frac{L(m,n)}{l(\mu_{\inf})}\to 1$.
\end{itemize}
\end{cor}

\begin{proof}
Note that as $H_m\to H$ weakly, $\mu_{(m)}\le \inf supp(H)$, and hence $\mu_{\inf}\le \inf supp(H)$.
If $\crit< \mu_{\inf}$, it is easy to check,
using the fact that $\mu_{\inf}=\displaystyle{\lim_{m\to\infty}} \inf supp(H_m)$
and using the analyticity of $l_m(z)$ for $z\in (0, \inf supp(H_m))$,
that $\crit_m\to \crit$, and also $\frac1{m}(l_m(z)-l(\crit))\to 0$ for each $z\in (0, \mu_{\inf})$.
Hence we are in the case (a) or (b) of Theorem~\ref{th:main} depending on $\alpha<\lambda<\mu_{\inf}$ or $\lambda<\alpha$, and the above results for (i) and (ii) follow.
On the other hand if $\crit>\mu_{\inf}$, then it is also easy to check that
for any compact interval of $(0, \mu_{\inf})$, there is no zero of the equation~\eqref{eq:solvecrit} for all large enough $m$. This means that $\displaystyle{\limsup_{m\to\infty}}  (\mu_{(m)}-\crit_m)\to 0$, and hence we find that the conditions for the case (c) of Theorem~\ref{th:main}  are satisfied.
It is also direct to check that $\frac1{m}(l^{(r)}_m(z)-l(z))\to 0$ for each $z\in (0, \displaystyle{\liminf_{m\to\infty}} \mu_{(m-r)})$ since $l^{(r)}_m(z)$ is analytic in any closed interval in this interval for all large enough $m$.
Thus the case (iii) follows.
\end{proof}

\begin{figure}[htbp]
\centering
\subfigure[]
{  \includegraphics[scale=0.25]{functionl2}
}
\subfigure[]
{  \includegraphics[scale=0.25]{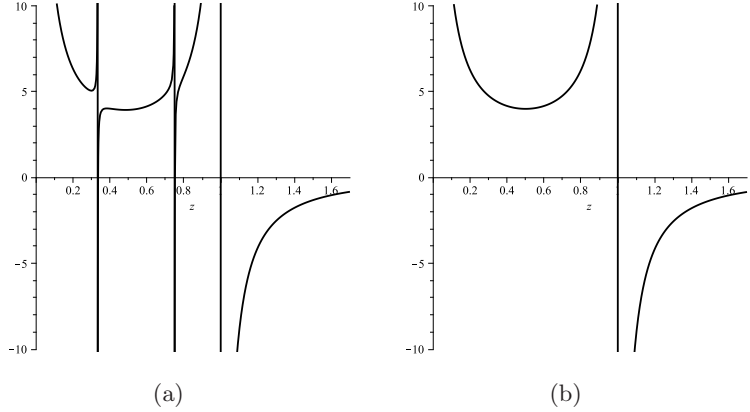}
}
\caption{We assume that $\mu_1=\frac13$, $\mu_2=\frac34$, and $\mu_i=1$ for $i\ge 3$.
Then $H_m= \frac1{m} \delta_{\frac13}+ \frac1{m}\delta_{\frac34}+ \frac{m-2}{m} \delta_{1}$
 and $H= \delta_1$. (a) is the graph of $\frac1{m} l_m(z)$ when $m=n=100$.
(b) is the graph of $\frac1{m} l(z)$ when $m=n=100$. }
\label{fig:functionlwithH}
\end{figure}

\subsection{Asymptotic result II. Second order}\label{sec:asy2}

The next theorem is about the second order asymptotics. We evaluate the law of the asymptotic fluctuations.
Let $\TWC$ denote the complex Tracy-Widom random variable from random matrix theory \cite{Tracy-Widom94}. There are other Tracy-Widom random variables, ${\sf TW}_\beta$, and the subscript $2$ in $\TWC$ signifies that
this random variable is related to the so-called complex case in random matrix theory. See~\eqref{eq:assm4} below for an explicit formula of $\TWC$ and Section \ref{sec:tw2} for additional information.
%The standard  normal random variable is denoted by $\mathcal{N}(0,1)$.
The notation $\overset{\mathcal{D}}{\longrightarrow}$ means convergence in distribution.

\begin{thm}[Asymptotic result II. Second order]\label{th:main2}
With the same notations and assumptions in Theorem~\ref{th:main},
we have the following asymptotic result.
\begin{enumerate}[(a)]
\item If $\alpha< \displaystyle{\liminf_{m\to\infty}} \crit_m$ and $\displaystyle{\liminf_{m\to\infty}} (\ser_{(m)}-\lambda_m)>0$, then
\begin{equation}\label{eq:fluc1}
\begin{split}
	\frac{L(m,n)-l_m(\crit_m)}{(l''_m(\crit_m)/2)^{1/3}} \overset{\mathcal{D}}{\longrightarrow} \TWC.
\end{split}
\end{equation}

\item If $\displaystyle{\limsup_{m\to\infty}} \crit_m<\alpha$, then
\begin{equation}\label{eq:fluc2}
\begin{split}
	\frac{L(m,n)-l_m(\alpha)}{(l'_m(\alpha))^{1/2}} \overset{\mathcal{D}}{\longrightarrow}  \mathcal{N}(0,1).
\end{split}
\end{equation}

\item Suppose that $\displaystyle{\limsup_{m\to\infty}}  (\ser_{(m)}-\lambda_m)=0$ and
that $\displaystyle{\liminf_{m\to\infty}}  (\ser_{(m-1)}-\ser_{(m)})>0$.
Furthermore, assume that $\displaystyle{\liminf_{m\to\infty}}  (\lambda^{(1)}_m-\mu_{(m)}) >0$ where
$\lambda^{(r)}_m$ and $l_m^{(r)}$(z) are defined in Theorem~\ref{th:main} (c).
Then
\begin{equation}\label{eq:fluc3}
\begin{split}
	\frac{L(m,n)-l_m^{(1)}(\mu_{(m)})}{(-(l_m^{(1)})'(\mu_{(m)}))^{1/2}} \overset{\mathcal{D}}{\longrightarrow}  \mathcal{N}(0,1).
\end{split}
\end{equation}

\end{enumerate}
\end{thm}

Note that the denominator in~\eqref{eq:fluc1} is $O(m^{1/3})$ while the denominators in~\eqref{eq:fluc2} and~\eqref{eq:fluc3} are $O(m^{1/2})$.
%Note that $(l''_m(\crit_m)/2)^{1/3}= O(m^{1/3})$ while $(l'_m(\alpha))^{1/2}=O(m^{1/2})$ and $(-(l^{(1)}_m)'(\mu_{(m)}))^{1/2}=O(m^{1/2})$.
Hence the order of the fluctuations in the case (a) is different from the cases (b) and (c).

\begin{rmk}\label{rmk:fluctuationc}
In the case (c), if
$\ser_{(m)}=\cdots= \ser_{(m-r+1)}$ and $\displaystyle{\liminf_{m\to\infty}}  (\ser_{(m-r)}-\ser_{(m-r+1)})>0$
for a fixed $r$,
then we have the same scaling but a different limiting distribution.
The new distribution is same as the largest eigenvalue of the so-called $r\times r$ matrix from the
Gaussian unitary ensemble \cite{Mehta04}.
\end{rmk}

The above theorem can be simplified as follows under some extra conditions.
The proof of this Corollary is similar to the argument in the proof of Corollary~\ref{rem:con23} and we skip it.

\begin{cor}[Easier conditions when $\mu_{(m)}\to \mu_{\inf}$]
% as $m\to \infty$ the slowest service rate converges to a constant.]
\label{rem:con23}
Suppose that
 $\mu_{(m)}\to \mu_{\inf}$ as $m\to \infty$ for some value $\mu_{\inf}$ as in Corollary~\ref{rem:con}.
Furthermore, assume that
\begin{equation}\label{eq:seccon1}
	l_m(z) -l(z) = o(m^{1/3}) \text{ for any compact interval of $z\in (0, \mu_{(m)})$}
\end{equation}
and
\begin{equation}\label{eq:seccon1}
	l_m^{(1)}(z) -l(z) = o(m^{1/2}) \text{ for any compact interval of $z\in (0, \mu_{(m-1)})$}
\end{equation}
where $l(z)$ is defined in~\eqref{eq:lz 2}.
Let $\lambda$ be defined in~\eqref{eq:solvecrit2}.
Then we have:
\begin{itemize}
\item[(i)] If $\lambda\in (\alpha, \mu_{\inf})$, then
\begin{equation}
	\frac{L(m,n)-l(\crit)}{(l''(\crit)/2)^{1/3}} \overset{\mathcal{D}}{\longrightarrow} \TWC.
\end{equation}
\item[(ii)] If $\lambda\in (0, \alpha)$, then
\begin{equation}
\begin{split}
	\frac{L(m,n)-l(\alpha)}{(l'(\alpha))^{1/2}} \overset{\mathcal{D}}{\longrightarrow}  \mathcal{N}(0,1).
\end{split}
\end{equation}

\item[(iii)] If $\lambda>\mu_{\inf}$ and $\displaystyle{\liminf_{m\to\infty}}  (\mu_{(m-1)}-\mu_{\inf}) >0$,
then
\begin{equation}
\begin{split}
	\frac{L(m,n)-l(\mu_{\inf})}{(-l'(\mu_{\inf}))^{1/2}} \overset{\mathcal{D}}{\longrightarrow}  \mathcal{N}(0,1).
\end{split}
\end{equation}

\end{itemize}
\end{cor}

%Theorems~\ref{th:main} and~\ref{th:main2} are proved in Section \ref{sec:proofs}.
%As mentioned before, the proof is based on a connection between  latency analysis and
%directed last passage percolation (DLPP) models. This connection is described in the next section.

\subsection{Related results from random matrix theory}

Phase transitions have also been studied in random matrix theory, especially for the so-called `spiked random matrix' models.  These models are of particular interest for principal component analysis in statistics \cite{Johnstone} and in this context the phase transition was first obtained in \cite{Baik-Ben_Arous-Peche05}.
By Theorem~\ref{thm:exact}, such results in random matrix theory can be interpreted as results for tandem queues.
The following cases have been studied:
\begin{enumerate}[(1)]
\item $\alpha=0$ and $\mu_1=\cdots= \mu_m$ in \cite{Johansson00},
\item $\alpha=0$ and $\mu_{r+1}=\cdots =\mu_m$ for a fixed number $r$ in \cite{Baik-Ben_Arous-Peche05,Baik-Silverstein06,Paul07,Benaych_Georges-Guionnet-Maida11,Benaych_Georges-Nadakuditi11},
\item $\mu_1-\alpha\to 0$,  $\mu_{r+1}=\cdots= \mu_m$ for a fixed number $r$
and $\mu_2, \cdots, \mu_r$ converge to a certain critical value in \cite{Ferrari-Spohn06,Borodin-Peche08a}.
\end{enumerate}
From Theorem~\ref{thm:exact}, one can see that the case when $\mu_1=\cdots= \mu_m$ and $\alpha>0$ is almost the same as the case when $\alpha=0$ and only $\mu_1$ is different from other service rates. Thus, case (2) above also includes the case when
\begin{enumerate}
\item[(4)] $\alpha>0$ and $\mu_1=\cdots= \mu_m$.
\end{enumerate}
For the above case, both the leading order and the second order asymptotics are known.
For general $\mu_i$'s, the following cases have been studied:
\begin{enumerate}
\item[(5)] the leading order term when $\alpha=0$ and $\mu_i$ are general in \cite{Nadakuditi-Silverstein10},
\item[(6)] the second order term when $\alpha=0$ and $\mu_i$ are general but satisfy the so-called sub-criticality condition, which yields a convergence to $\TWC$, in \cite{El_Karoui07}.
\end{enumerate}
%The case (6) corresponds to the case (a) with $\alpha=0$ in the above theorems.
%The other cases (b) and (c) are new in this paper.
%Moreover, in this paper we consider the more general setting where  $\alpha > 0$ and $\mu_{1}, \ldots, \mu_{m}$ are arbitrary. This general case does not follow from existing results in the literature.

In this paper we obtain the asymptotic results for the general setting where  $\alpha > 0$ and $\mu_{1}, \ldots, \mu_{m}$ are arbitrary. This general case does not follow from existing results in the literature.

\section{Examples and discussions}\label{sec:examples}

\begin{example}[Equal  service rates, initially empty queues]\label{ex:null0}
Suppose $\ser_1=\cdots=\ser_m=:\mu$ for a fixed $\mu$ and  $\alpha=0$.
Clearly Corollary~\ref{rem:con} applies with $\mu_{\inf}=\mu$, $H=\delta_{\mu}$,
and $l(z)= \frac{m}{\mu-z}+ \frac{n}{z}$.
(We also have $H_m=\delta_{\mu}$.)
The equation~\eqref{eq:solvecrit2} becomes
$\frac{m}{(\mu-\crit)^2}-\frac{n}{\crit^2} =0$ for $\crit\in (0, \mu)$.
Solving this algebraic equation we obtain
\begin{equation}\label{eq:equal1}
\begin{split}
	\crit= \frac{\sqrt{n}}{\sqrt{m}+\sqrt{n}} \mu.
\end{split}
\end{equation}
Since $\lambda< \mu=\mu_{\inf}$, case (iii) does not occur.
Case (ii) also does not occur since $\alpha=0$.
Therefore from case (i) we find that
\begin{equation}
	L(m,n) \approx
	\frac{(\sqrt{m}+\sqrt{n})^2}{\mu}
\end{equation}
as $m,n\to \infty$ such that $m/n\to \gamma\in (0,\infty)$.
\end{example}

\begin{example}[Equal  service rates, queues in equilibrium]\label{ex:null}
The only change from the previous example is that $\alpha>0$. Hence case (ii) may occur and we find that
\begin{equation}\label{eq:example2}
\begin{split}
	L(m,n) \approx \begin{cases}
	\frac{(\sqrt{m}+\sqrt{n})^2}{\mu} \qquad\quad &\text{if $\alpha< \frac{\sqrt{n}}{\sqrt{m}+\sqrt{n}} \mu$,} \\
	\frac{m}{\mu-\alpha}+\frac{n}{\alpha}  \qquad\quad &\text{if $\alpha> \frac{\sqrt{n}}{\sqrt{m}+\sqrt{n}} \mu$.}
	\end{cases}
\end{split}
\end{equation}
\end{example}

Example~\ref{ex:null}  shows a phase transition phenomenon.
When $\alpha$ is below the threshold $\frac{\sqrt{n}}{\sqrt{m}+\sqrt{n}} \mu$, the latency $L(m,n)$ is same as that of the system with initially empty queues. The existing customers do not contribute to the leading order asymptotic of $L(m,n)$. (Actually they do not contributed even to the second leading order asymptotic; see Corollary~\ref{rem:con23} and Example~\ref{ex:spikedfluc}.)
On the other hand, if $\alpha$ is above the threshold, then we see the effect of the existing customers.
Note that the first customer in the batch of $n$ customers sees a geometrically distributed random number of customers in each queue since the queues are in equilibrium.
This results that the time for this customer stays in each queue is exponential of rate $\mu-\alpha$.
Hence the expectation of the time it takes this customer exit from all the queues is $\frac{m}{\mu-\alpha}$.
The asymptotic formula $\frac{m}{\mu-\alpha}+\frac{n}{\alpha}$ in~\eqref{eq:example2}  is larger than this number.
This is natural since $n$ is same order as $m$ and therefore the other $n-1$ customers contribute to the leading order asymptotic of $L(m,n)$.
%It turned out that the combined effect of the existing customers and the newly arrived customers yield $\frac{m}{\mu-\alpha}+\frac{n}{\alpha}$.
See Section~\ref{sec:intuition} for an intuitive reason for the value of the critical threshold for $\alpha$
and the asymptotic latency $\frac{m}{\mu-\alpha}+\frac{n}{\alpha}$.

In the next examples, we assume that $\alpha\ge 0$.

\begin{example}[Quantized service times]\label{ex:quant}
Let $f(y)$ be a piecewise-continuous, non-negative and compactly support function with total mass $1$
Suppose that $\mu_i= F^{-1}(\frac{i}{m})$ for $F(x):= \int_{-\infty}^xf(y) dy$.
Then $\mu_{(m)}\to \mu_{\inf}= \sup\{x: F(x)=0\}$ and $dH(x)=f(x) dx$.
Thus Corollary~\ref{rem:con} applies and case (iii) does not occur.
Let $\crit$ be the unique solution to the equation
\begin{equation}
	m\int \frac{f(y)}{(y-\crit)^2}dy - \frac{n}{\crit^2} = 0 , \qquad \crit \in (0,  \mu_{\inf}).
\end{equation}
Then
\begin{equation}
\begin{split}
	L(m,n) \approx \begin{cases}
	m\int \frac{f(y)}{y-\crit}dy + \frac{n}{\crit}  \qquad\quad &\text{if $\alpha< \crit$,} \\
	m\int \frac{f(y)}{y-\alpha}dy + \frac{n}{\alpha}  \qquad\quad &\text{if $\alpha> \crit$.}
	\end{cases}
\end{split}
\end{equation}
\end{example}

\begin{example}[One slow server]\label{ex:spiked1}
Suppose that all service rates are equal except for one server whose rate is smaller than the rest.
%We consider the case when  one of the rate is smaller than the rest which are all equal.
From Remark~\ref{rmk:ordering}, we may assume that the slow server is  the first server  without loss of generality.
Set $\ser_1=\mu'$ and $\ser_2=\cdots=\ser_m=:\mu$ where $\mu, \mu'$ are fixed numbers such that
$\alpha<\mu'< \mu$.
Then as $\mu_{(m)}=\mu'$ is a constant, Corollary~\ref{rem:con} applies with $\mu_{\inf}=\mu'$.
We have $H_m= \frac1{m} \delta_{\mu'}+(1-\frac1{m})\delta_{\mu}\to H=\delta_{\mu}$.
This is same $H$ as in the Example~\ref{ex:null0}.
Thus we find that $\crit= \frac{\sqrt{n}}{\sqrt{m}+\sqrt{n}} \mu$ as before.
The difference is that since $\mu_{\inf}=\mu'$ is different from $\mu$, the case (iii) can occur.
We obtain
\begin{equation}\label{eq:oneslowexa}
\begin{split}
	L(m,n) \approx \begin{cases}
	\frac{(\sqrt{m}+\sqrt{n})^2}{\mu} \qquad\quad &\text{if $\alpha< \frac{\sqrt{n}}{\sqrt{m}+\sqrt{n}} \mu< \mu'$,} \\
	\frac{m}{\mu-\alpha}+\frac{n}{\alpha}  \qquad\quad &\text{if $ \frac{\sqrt{n}}{\sqrt{m}+\sqrt{n}} \mu<\alpha$,}\\
	%<\mu'$,}\\
	\frac{m}{\mu-\mu'}+ \frac{n}{\mu'}  \qquad\quad &\text{if $%\alpha<
	\mu'< \frac{\sqrt{n}}{\sqrt{m}+\sqrt{n}} \mu$.}
	\end{cases}
\end{split}
\end{equation}
{Figure \ref{fig:simulation} numerically validates the theoretical predictions.}
\end{example}

\begin{figure}
\centering
\includegraphics[width=3.6in,trim = 260 20 250 0, clip = true]{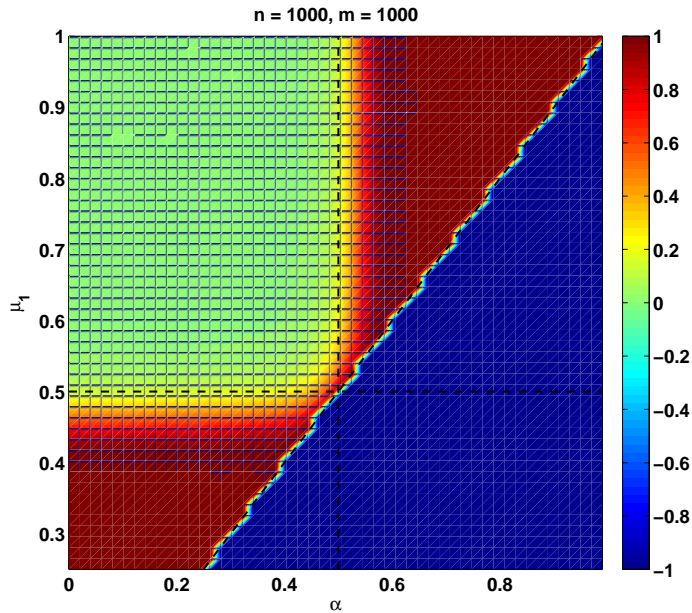}
\caption{A heat-map of the Kolmogorov-Smirnov (KS) distance between the empirical cdf's of $1000$ realizations of the $L(m,n)$ statistic of two systems both of which have $n  = m = 1000$ and $m-1$ servers with rates $\mu_{2} = \ldots= \mu_{m}$. The baseline system  has $\mu_{1} = 1$ and the external arrival rate $\alpha = 0$. The active system has $\mu_{1}$ and $\alpha$ that are varied in the range specified by the plot. Using the KS distance we can assess if the underlying probability distributions for the $L(m,n)$ statistic for the two systems described differ. A value close to $0$ indicates  that the distributions are `near' while a value closer {to} $1$ indicates that they are `far'. Note the phase transitions that separate regimes where the distributions are similar from regimes where the distributions are different; the dashed lines correspond to the predicted value of the phase {transition} (see  Example \ref{ex:spiked1}). The portion of the heat-map where the KS distance equals $-1$ corresponds to the (inadmissible) setting where $\alpha > \mu$.}
\label{fig:simulation}
\end{figure}

\begin{figure}
\centering
\subfigure[Subcritical regime: $n = m = 100$, $\mu_1, \ldots \mu_m =1$, $\alpha = 0.3$]{
\includegraphics[width=5.0in,trim = 75 75 75 50, clip = true]{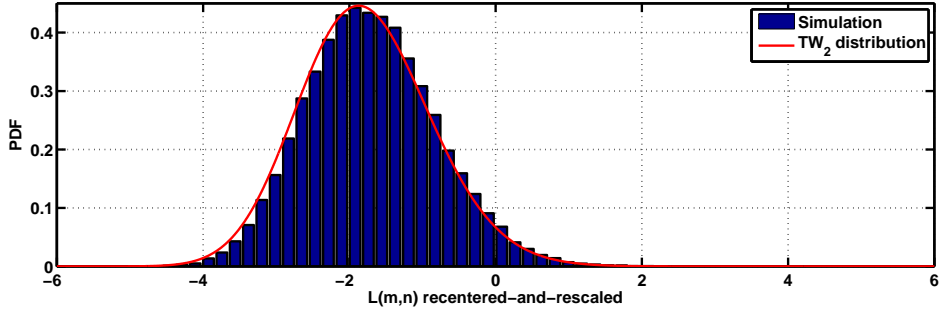}}
 \\
\subfigure[Supercritical regime: $n = m = 100$, $\mu_1 , \ldots, \mu_m = 1$, $\alpha = 0.7$.]{
\includegraphics[width=5.0in,trim = 75 75 75 50, clip = true]{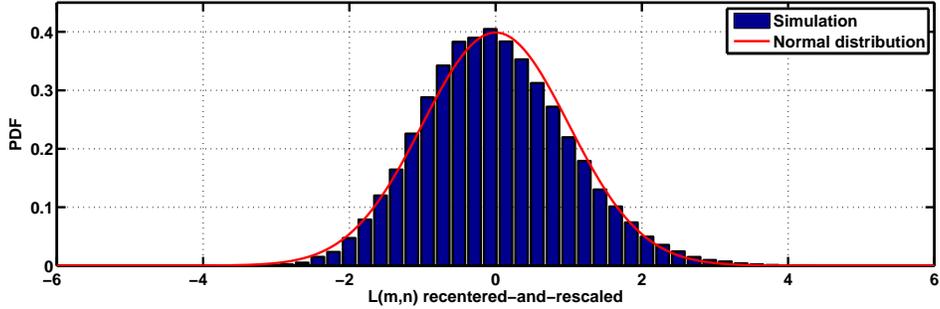}}
\caption{Comparison of the empirical pdf's computed over $50,000$ realizations of the $L(m,n)$ statistic of two systems both of which have $n  = m = 100$ and all $m$ servers with rates $\mu_{1} = \ldots= \mu_{m} =1$. In (a), the external arrival rate $\alpha = 0.3$ and we compare the empirical distribution of the $L(m,n)$ statistic, rescaled and recentered as described in Example \ref{ex:spikedfluc}, with the pdf of the $\TWC$ distribution, computed as described in Section \ref{sec:tw2}. In (b), the external arrival rate $\alpha = 0.7$ and we compare the empirical  distribution of the $L(m,n)$ statistic, rescaled and recentered as described in Example \ref{ex:spikedfluc}, with the pdf of the normal distribution. See Section \ref{sec:tw2} for a discussion of the properties of the $\TWC$ distribution; Figure \ref{fig:twnormal} compares the $\TWC$ distribution with the normal distribution.}
\label{fig:supersub example}
\end{figure}

This example illustrates another phase transition.
The last equation in~\eqref{eq:oneslowexa} shows that the leading order asymptotic of $L(m,n)$ changes if the rate $\mu'$ of the slow server is below a threshold.
If this rate is small enough, i.e. below the threshold $\frac{\sqrt{n}}{\sqrt{m}+\sqrt{n}} \mu$, then the time spent by the customers in this queue is significant large and it affects the total exit time.
However, if the rate is above the threshold, the effect of the single slow server is not detectable in terms of the leading order (and second order indeed) asymptotic of $L(m,n)$, as seen in the first two equations in~\eqref{eq:oneslowexa}.
Note that this threshold is same as the one for $\alpha$.

\begin{example}[Two slow servers]\label{ex:spiked2}
Assume $\ser_1=\mu'$, $\ser_2=\mu''$, and $\ser_3=\cdots=\ser_m=:\mu$ where $\alpha<\mu'< \mu''< \mu$.
Then  we can apply Corollary~\ref{rem:con}, and we have $H=\delta_{\mu}$, $\mu_{\inf}=\mu'$, and
$\crit= \frac{\sqrt{n}}{\sqrt{m}+\sqrt{n}} \mu$.
Using Corollary~\ref{rem:con}, we have exactly the same result as~\eqref{eq:oneslowexa}.
Note that for the case (iii) in which the slow servers affect the batch latency, the leading order of the exit time depends only on the lowest service rate, not the second lowest service rate.
It is easy to check that~\eqref{eq:oneslowexa} also holds if $\mu'=\mu''<\mu$ or if there are three slow servers, etc.
\end{example}

{
Hence the two key parameters are $\alpha$ and the rate of the slowest server.

}

\begin{example}[Second order asymptotics]\label{ex:spikedfluc}
For  Examples~\ref{ex:spiked1} and~\ref{ex:spiked2}, Corollary \ref{rem:con23} implies that
\begin{equation}
\begin{split}
	L(m,n) \approx \begin{cases}
	\frac{(\sqrt{m}+\sqrt{n})^2}{\mu} +\frac{\sqrt{m}+\sqrt{n}}{\mu} \big( \frac1{\sqrt{m}}+\frac1{\sqrt{n}} \big)^{1/3}  \TWC
\quad &\text{if $\alpha< \frac{\sqrt{n}}{\sqrt{m}+\sqrt{n}} \mu< \mu'$,} \\
	\big(\frac{m}{\mu-\alpha}+\frac{n}{\alpha}\big) +  \big(\frac{m}{(\mu-\alpha)^2}-\frac{n}{\alpha^2}\big)^{1/2}\mathcal{N}(0,1) \quad &\text{if $ \frac{\sqrt{n}}{\sqrt{m}+\sqrt{n}} \mu<\alpha$,}\\
	\big( \frac{m}{\mu-\mu'}+ \frac{n}{\mu'}\big) +  \big( \frac{n}{(\mu')^2}-\frac{m}{(\mu-\mu')^2}  \big)^{1/2} \mathcal{N}(0,1)
 \quad &\text{if $\mu'< \frac{\sqrt{n}}{\sqrt{m}+\sqrt{n}} \mu$.}
	\end{cases}
\end{split}
\end{equation}
\end{example}

{Figure \ref{fig:supersub example} validates the distributional characterization in Example \ref{ex:spikedfluc} for the setting where $\mu_1 = \ldots = \mu_m = 1$ and $n = m$ so that when $\alpha < 0.5$ we expect the $\TWC$ distribution whereas when $\alpha> 0.5$ we expect a standard normal distribution for the recentered and rescaled $L(m,n)$ distribution.}

\section{Proof of Theorem~\ref{thm:exact}}\label{sec:proofexact}

In a  tandem of queues, let $w(i,j)$ denote the service time for customer $j$ at server $i$.
Here we label the customers in the order that they exit the system and label the servers in  the usual order.
The exit time $L(i,j)$ of customer $j$ from queue $i$ satisfies a recursive relation
\begin{equation}\label{eq:lpp recursion}
	L(i,j)= \max\{L(i-1,j), L(i,j-1)\} + w(i,j), \qquad i,j\in \Z.
\end{equation}
This  recursion can be found in the {seminal} paper by Glynn and Whitt \cite{Glynn-Whitt91} wherein the authors attribute the formulation to Tembe and Wolff \cite{Tembe-Wolff74}.
Note that this recursion %is a deterministic relation and
holds for any service times $w(i,j)$.

The above recursion can be recast \cite{Glynn-Whitt91} as the following directed last passage percolation (DLPP) problem :
\begin{equation}\label{eq:lpp wij}
	L(i,j)= \max_{\pi \in P(i,j)} \bigg( \sum_{(k,\ell)\in \pi} w(k,\ell) \bigg),
\end{equation}
where $P(i,j)$ is the set of `up/right paths' ending at $(i,j)$ \textit{i.e.} $\pi\in P(i,j)$ if
$\pi = \{(k_{s}, \ell_s)\}_{s=-\infty}^0$ such that $(k_0, \ell_0)=(i,j)$ and $(k_s, \ell_s)-(k_{s-1}, \ell_{s-1})$ is either $(1,0)$ or $(0,1)$ for all $s\le 0$.
Here we take $w(k, \ell)=0$
if there is no customer $k$ or if the customer $k$ starts at the server labeled larger than $\ell$ at the initial time so that the customer $k$ does not need any service from the server $\ell$.
The above identity can be obtained by noting that the right-hand-side of (\ref{eq:lpp wij}) satisfies the same recurrence as $L(i,j)$ in (\ref{eq:lpp recursion}).
%since a path in $P(i,j)$ consists of either a path in $P(i-1,j)$ and $(i,j)$, or a path in $P(i,j-1)$ and $(i,j)$.
In the context of DLPP, $L(m,n)$ is referred as the `last passage time' to the site $(m,n)$ and $w(i,j)$ are called weights (see, for example, \cite{Johansson00,Ferrari-Spohn,Corwin12}).

Certain DLPP models with special weights are called `solvable' and they are defined as follows.

\begin{defn}[Solvable DLPP models]
Let  $a_i$ and $b_i$, $i=1, 2, \cdots$, be the real numbers in $(-\infty, \infty]$ such that $a_i+b_j>0$ for all $i,j\ge 1$.
For $i,j \ge 1$, let $w(i,j)$ be independent exponential random variables with rate $a_i+b_j$ (so that the mean is $1/(a_i+b_j)$.)
For other $i,j$ in $\Z^2$, we set $w(i,j)=0$.
%If $a_i+b_j=\infty$,  we regard $w(i,j)=0$.
The DLPP model with these weights are defined as solvable.
\end{defn}

{
A special property of these solvable DLPP models is that they are related to certain random matrices.
The following result is due to Borodin and P\'ech\'e \cite{Borodin-Peche08a}.
The case when all $a_i$ and $b_i$ are same was first obtained  by Johansson in \cite{Johansson00}.
}

\begin{prop}[\cite{Borodin-Peche08a}]\label{prop:bp}
Let $L(m,n)$ be the last passage time for the solvable DLPP model.
On the other hand,
let $X$ be an $m\times n$ matrix with independent entries distributed as
\begin{equation} \label{eq:rmt aibj}
	X_{ij} \sim \mathcal{CN}\left(0, \frac{1}{a_i+b_j} \right).
\end{equation}
Then
\begin{equation}
L(m,n) \overset{\mathcal{D}}{=} \Lambda_{\max}(XX^{*}),\qquad m,n\ge 1,
\end{equation}
where $\Lambda_{\max}$ denotes the largest eigenvalue.
\end{prop}

{
Another special property of the solvable DLPP models is that the cdf of $L(m,n)$ can be obtained explicitly for all $m, n$. This will be stated in Proposition~\ref{prop:FredDet}.
}

Our setting of tandem in equilibrium can be thought of a special case of the solvable DLPP problem as follows.
Note that the exit time of a customer who arrives via an external Poissonian arrival process of rate $\alpha$ at an $M/M/1$ queue in equilibrium with service rate $\ser_i$ (where $\alpha<\mu_i$) is exponentially distributed with rate $\ser_i - \arr$. Moreover, if the same customer arrives at a tandem of $m$ $M/M/1$ queues in equilibrium, then
the exit time of the customer from all queues is distributed as $X_1+\cdots + X_m$ where $X_i$ are independent exponential random variables with rate $\ser_i-\arr$.
{This means that it is as if the customer arrives at a tandem of $m$ queues which are all empty but the service rate is changed to $\ser_i-\arr$.
Hence if a batch of $n$ customers arrive %via a Poisson arrival process
to a tandem of $m$ queues in equilibrium,
then in terms of the latency $L(m,n)$ it is as if the queues are initially all empty, the service rates for the first of the $n$ customers are $\ser_i-\arr$, and the service rates for the other $n-1$ following customers are $\ser_i$.
%then the exit time $L(m,n)$ of all $n$ customers has the same distribution as the exit time in the case when
%the queues are initially empty, the service rates are $\ser_i$ for all but the first of the $n$ customers for whom the service rates are $\ser_i-\arr$:
Thus the weight $w(i,j)$ is exponential with rate $\ra_i$ for $i=1, \ldots, m$, $j = 2, \ldots, n$ while
$w(i,1)$ is exponential with rate $\ser_i-\arr$ for $i=1, \ldots, m$. All other $w(i,j)=0$.
Hence,
}

\begin{prop}\label{prop:aa}
The latency $L(m,n)$ for the model in Section~\ref{sec:main results} is distributed as the last passage time of the solvable DLPP model with
%$a_i =\ra_i$ for $i=1, \ldots, m$, $b_1= -\arr$, $b_j=0$ for $j=2, \ldots, n$, and all other $a_i=b_j=\infty$.
\begin{equation}\label{eq:eqweights}
\begin{split}
	a_i =\ra_i, \quad \text{for $i=1, \ldots, m$}, \qquad
	        b_1 = -\arr,  \qquad
	        b_j  = 0,  \quad  \text{for $j=2, \ldots, n$,}
\end{split}
\end{equation}
and all other $a_i=b_j=\infty$.
\end{prop}

The same idea was used in \cite{Prahofer-Spohn02a} to map an interacting particle system
starting in equilibrium to a solvable DLPP model.
%, called the totally asymmetric simple exclusion process, starting in equilibrium to a solvable DLPP problem.
%DLPP is also related to random growth models and has been extensively studied in both physics and mathematics communities (see, for example, \cite{Johansson00}, \cite{Corwin12}, and references therein.)

Theorem~\ref{thm:exact} now follows from this proposition and
Proposition~\ref{prop:bp}.

\medskip

{
We note that there are a few results on the asymptotics of the last passage time for i.i.d. weights other than exponential random variables:  \cite{Glynn-Whitt91,Seppalainen97a,Baik-Suidan05,Bodineau-Martin05,Suidan06,Hambly-Martin}.
These results can be interpreted as the results for the latency of a batch of customers in a tandem of queues which are initially empty and the service rates are i.i.d.
}

\section{Intuition for the critical threshold}\label{sec:intuition}

The value of the critical threshold in the Examples in Section~\ref{sec:examples} can be understood 
using the DLPP formulation if we already know the asymptotic of $L(m,n)$ when $\alpha=0$. 
Consider the setup of Example~\ref{ex:null} when all $\mu_i$ are $\mu$. 
Then the rate of $w(i,j)$ for DLPP model is $\mu-\alpha$ for $j=1$ and is $\mu$ for $j\ge 2$. 
Let us use the notation $L_\alpha(m,n)$ in order to indicate the dependence on $\alpha$. 
Suppose we already know that $L_0(m,n)\approx \frac{(\sqrt{m}+\sqrt{n})^2}{\mu}$. 
Now a path $\pi$ in the formula~\eqref{eq:lpp wij} consists of two parts, one from $(1,1)$ to some point $(p,1)$ and the other from $(p,2)$ to $(m,n)$. See Figure~\ref{fig:DLPP} (a). Here $p$ may be $1$. 
\begin{figure}[htbp]
\centering
\subfigure[]
{  \includegraphics[scale=0.4]{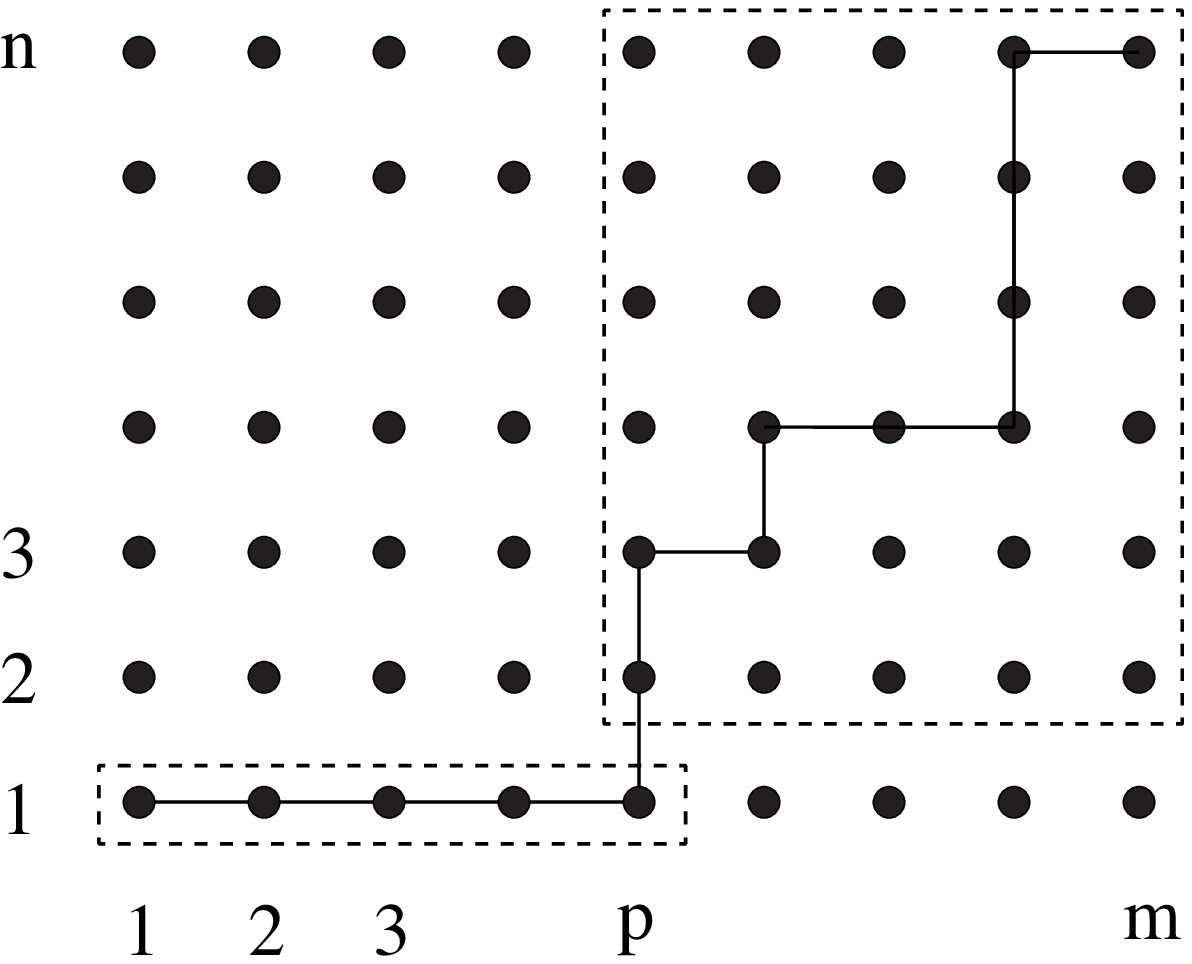}
} \qquad 
\subfigure[]
{  \includegraphics[scale=0.4]{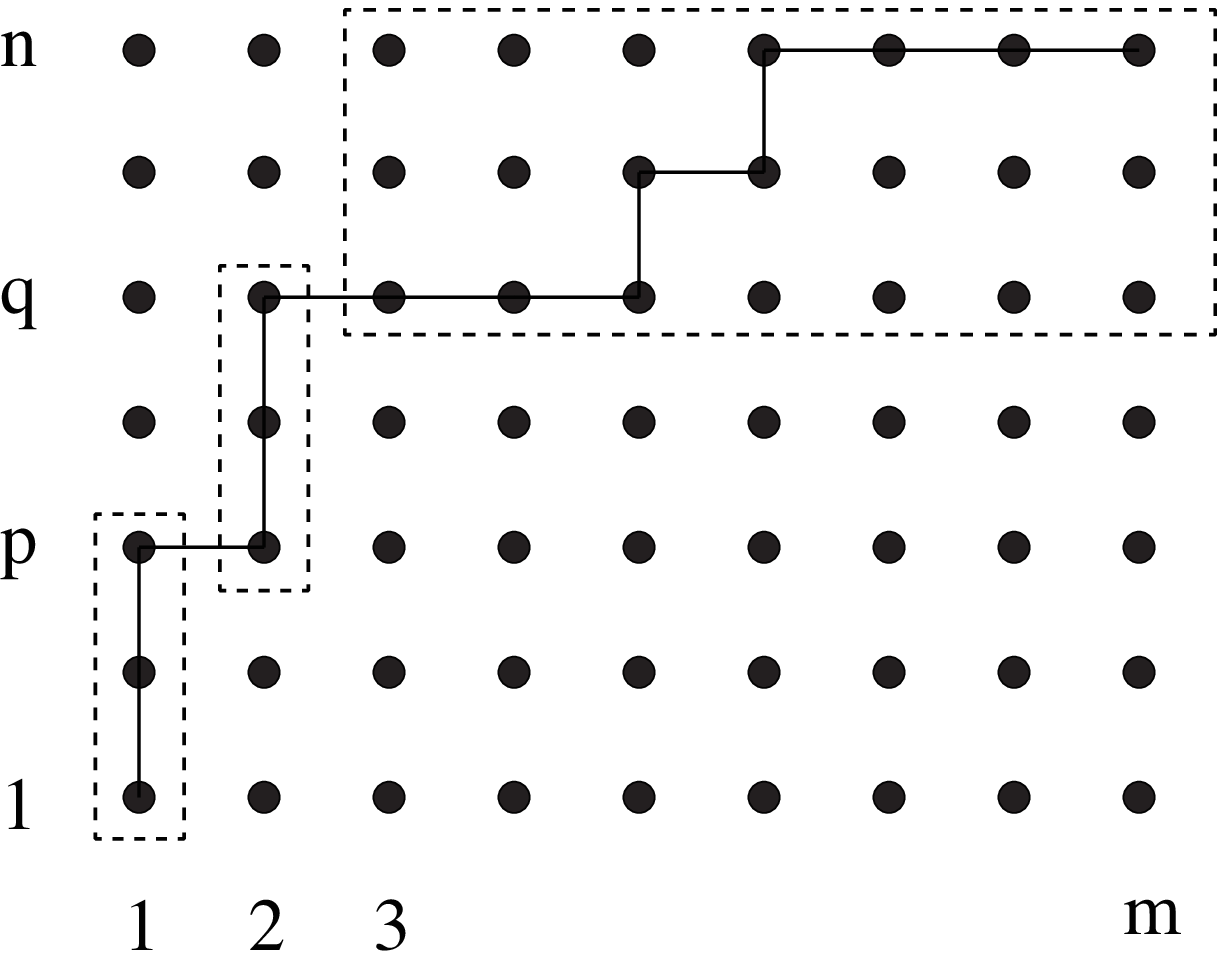}
}
\caption{(a) An up/right path $\pi$ for Example~\ref{ex:null}. 
(b) An up/right path for Example~\ref{ex:spiked2} when $\alpha=0$}
\label{fig:DLPP}
\end{figure}

The expectation of the sum of $w(k, \ell)$ along the first part is 
$\frac{p}{\mu-\alpha}$ using the sum of independent exponential random variables.
On the other hand, the second part lies in a rectangle of size $(m-p+1)\times (n-1)$ where all rates are $\mu$. 
If we choose the maximal path in this part, its expectation is $\frac{(\sqrt{m-p+1}+\sqrt{n-1})^2}{\mu}$ using the formula for $L_0$. 
Now since $p$ can be any point, we expect that 
\begin{equation}\label{eq:maxtem1}
\begin{split}
	L_\alpha(m,n)
	&= \max_{p\in \{1, \cdots, m\}} 
	\bigg( \frac{p}{\mu-\alpha} + \frac{(\sqrt{m-p+1}+\sqrt{n-1})^2}{\mu} \bigg) \\
	&\approx \max_{0\le x\le 1} 
	\bigg( \frac{mx}{\mu-\alpha} + \frac{(\sqrt{m(1-x)}+\sqrt{n})^2}{\mu} \bigg).
\end{split}
\end{equation}
It is easy to check that the maximum occurs at $x=0$ if $\alpha\le \crit$
and occurs at $x=1- \frac{n (\mu-\alpha)^2}{m \alpha^2}\in (0,1)$ if $\alpha>\lambda$
where $\crit= \frac{\sqrt{n}}{\sqrt{m}+\sqrt{n}}\mu$.
This gives the threshold for $\alpha$. 
The maximum is given by 
$\frac{(\sqrt{m}+\sqrt{n})^2}{\mu}$ in the first case and by $\frac{m}{\mu-\alpha}+\frac{n}{\alpha}$
in the second case which is consistent to~\eqref{eq:example2}.
This is same argument given in Section 6 of \cite{Baik-Ben_Arous-Peche05}.

The threshold for the rate of the slow servers in Example~\ref{ex:spiked1} and Example~\ref{ex:spiked2} can be understood in a similar way. In the setting of Example~\ref{ex:spiked2}, when $\alpha=0$, the up/right path consists of three parts as indicated in Figure~\ref{fig:DLPP} (b), and we are lead to 
\begin{equation}\label{eq:maxtem}
\begin{split}
	L(m,n)
	&\approx \max_{0\le x\le y\le 1} 
	\bigg( \frac{nx}{\mu'} +\frac{n(y-x)}{\mu''} + \frac{(\sqrt{m}+\sqrt{n(1-y)})^2}{\mu} \bigg).
\end{split}
\end{equation}
Recall that $\mu'<\mu''<\mu$. 
A direct calculation shows that the maximum occurs when $x=y$ and the maximum does not depend on $\mu''$. 
This shows that the latency depends only on the slowest rate, not the other slow rates. 
If $\mu'>\frac{\sqrt{n}}{\sqrt{m}+\sqrt{n}}\mu$, 
then the maximum is $\frac{(\sqrt{m}+\sqrt{n})^2}{\mu}$ which is attained at $x=y=0$, 
and if $\mu'<\frac{\sqrt{n}}{\sqrt{m}+\sqrt{n}}\mu$, then the maximum is $\frac{m}{\mu-\mu'}+ \frac{n}{\mu'}$ which is attained at $x=y=1-\frac{m}{n} ( \frac{\mu'}{\mu-\mu'})^2$.
This calculation is consistent to Example~\ref{ex:spiked2}. 
If $\alpha>0$, then one needs to take a maximum of~\eqref{eq:maxtem} and~\eqref{eq:maxtem1}.

It is possible to make the above intuitive argument rigorous for all the examples discussed in Section~\ref{sec:examples}. 
However since we would like to prove general asymptotic theorems for much wider settings, we do not attempt to use the above argument and use a different approach discussed in the next section. 

\section{Proof of  Theorems~\ref{th:main} and~\ref{th:main2}} \label{sec:proof2}

We now prove Theorems~\ref{th:main} and~\ref{th:main2}.
Note that Theorem~\ref{th:main} follows from Theorem~\ref{th:main2}, except for the case (c) when $r>1$.
This case follows from Remark~\ref{rmk:fluctuationc} instead for which we comment at the end how the proof should be modified.
Hence we focus on the proof of Theorem~\ref{th:main2}.

\subsection{Formula of cumulative distribution function}\label{sec:cdffo}

The starting point of the asymptotic analysis is an explicit formula of the cdf of $L(m,n)$.
%This explicit formula was established for the general solvable DLPP models defined in Proposition~\ref{prop:bp}
%and propelled the exciting advancement of the theory of DLPP and random growth models in the last decade.
We state the cdf for general solvable model and then specialize it to our case.
Recall the parameters $a_i, b_j$ in Proposition~\ref{prop:bp}.
Define the kernel
\begin{equation}\label{eq:Kdouble}
\begin{split}
	K(\xi, \eta) = \frac{-1}{(2\pi)^2} \int \int
	\bigg( \prod_{i=1}^m \frac{a_i-w}{a_i-z} \bigg)
	\bigg( \prod_{j=1}^n  \frac{b_j+z}{b_j+w} \bigg)
	 \frac{e^{\eta w-\xi z}}{w-z} dz dw, \quad \xi, \eta\in \R.
\end{split}
\end{equation}
Here the contours are simple closed curves in the complex plane $\C$, oriented counter-clockwise, such that the contour of $z$ contains all $a_i$'s inside, the contour of $w$ contains all $-b_j$'s inside, and they do not intersect.
See Figure~\ref{fig:contour}. Note that the condition $a_i+b_j>0$ guarantees that such contours exist.
\begin{figure}[htbp]
\centering
\includegraphics[scale=0.6]{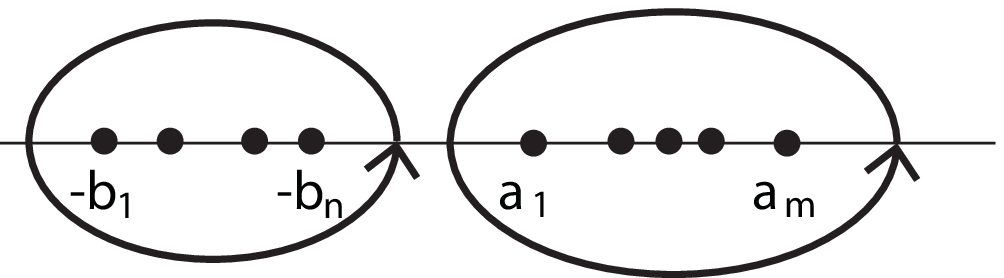}
\caption{Contours for~\eqref{eq:Kdouble}}
\label{fig:contour}
\end{figure}
%Alternatively, we may also take the contours as straight lines parallel to the imaginary axis such that the $z$-curve is oriented from bottom to top,  the $w$-curve is oriented from top to bottom and they satisfy $\max\{-b_j\}<Re(w)< Re(z)< \min\{a_i\}$.
%In~\eqref{eq:Kdouble},
%If $a_i=\infty$ for some $i$,  we interpret that $\frac{a_i-w}{a_i-z}=1$. The same remark also applies for $b_j$'s.
Now let  $K_x$ be the integral operator on the $L^2((x,\infty))$ defined by the kernel $K(\xi, \eta)$:
\begin{equation}
\begin{split}
	(K_x f)(\xi) = \int_x^\infty K(\xi, \eta) f(\eta) d\eta, \qquad \text{for $f\in L^2((x,\infty))$.}
\end{split}
\end{equation}
Then we have:

\begin{prop}[\cite{Johansson00}, \cite{Okounkov01a}]\label{prop:FredDet}
For $x>0$, %we have
\begin{equation}\label{eq:PKxL}
\begin{split}
	\mathbb{P}\{ L(m,n)\le x\} = \det(1-K_x).
\end{split}
\end{equation}
\end{prop}

Here the determinant is the Fredholm determinant of the operator $K_x$.
It can be expressed explicitly as
\begin{equation}\label{eq:Frdetdef}
\begin{split}
	\det(1-K_x) = 1 + \sum_{k=1}^\infty \frac{(-1)^k}{k!} \int_{(x, \infty)^k} \det\big( K(\xi_i, \xi_j\big))_{i,j=1}^k  d\xi_1 \cdots d\xi_k.
%	- \int_{x}^\infty  K(\xi, \xi) d\xi + \frac1{2!}  	\int_x^\infty \int_x^\infty \det(K(\xi_i, \xi_j))_{i,j=1}^2  d\xi_1 d\xi_2
\end{split}
\end{equation}
However, we do not use this expansion here. 
We only use the fact that
if a sequence of operators $T_n$ converges to $T$ in trace norm, then $\det(1-T_n)$ converges to $\det(1-T)$.
When the operators are given by kernels, as in our case, the convergence in trace norm is obtained if the kernels and their derivatives converge. See, for example, chapter 3 and 4 of \cite{Simon05}.
%See, for example, \cite{Simon05} for more details .

%\begin{proof}
Proposition~\ref{prop:FredDet} was first obtained by Johansson in \cite{Johansson00} when
all $a_i$'s and $b_j$'s are equal.
%This was done first by noting that the exponential weight is a simple limit of geometric weight
%and then showing that the DLPP model with iid geometric weights can be mapped to a certain probability measure on the set of partitions.
For general $a_i$ and $b_j$, the result was obtained for the geometric weights by Okounkov \cite{Okounkov01a} (see Section 2.2.3).
%The result of \cite{Okounkov01a} is given for the so-called Schur measure on partitions, but it is easy to see that it is same as the DLPP with geometric weights.
A simple limit to exponential weights of geometric weights yield the above proposition.
A good place where this is summarized explicitly is Theorem 3 of \cite{Borodin-Peche08a} (where one should set $r=s=p=1$).
%We note that Proposition~\ref{prop:bp} is obtained by evaluating the cdf of
%$\lambda_{\max}(XX^{*})$ and then noting that it is same as that of $L(m,n)$.
%A more direct map between solvable DLPP and the random matrix is yet to be found.
%For our case, we insert
For the choice of parameters~\eqref{eq:eqweights} of our case,
\begin{equation}\label{eq:K1-1}
\begin{split}
	K(\xi, \eta)
	&= \frac{-1}{(2\pi)^2} \int_{\Sigma_{ \{\alpha, 0\} } } \int_{\Sigma_{\{\mu_i\}_{i=1}^m} }
	\bigg(  \prod_{i=1}^m \frac{\mu_i-w}{\mu_i-z} \bigg)
	 \frac{z-\alpha}{w-\alpha}  \bigg( \frac{z}{w} \bigg)^{n-1}
	 \frac{e^{\eta w-\xi z}}{w-z} dz dw %\\
%	 & = \frac{-1}{(2\pi)^2} \int \int 	 \frac{e^{ m F(z,w)}}{w-z} dz dw
\end{split}
\end{equation}
where $\Sigma_A$ denotes a simple closed, counter-clockwise contour in $\C$ which encloses the points in the set $A$.

Theorem~\ref{th:main2} can now be obtained if we show that the kernel~\eqref{eq:K1-1} converges, after appropriate scaling, to the Airy kernel as $m,n \to\infty$.
This is done by applying the method of steepest-descent.
Some general references for the method of steepest-descent analysis are \cite{Erdelyi} \cite{Miller}.
The idea of steepest-descent method is to find the contour so that the real part of the exponent of the integrand has a unique maximum so that the contribution to the integral in the large $m,n$ limit is obtained from a small neighborhood of this maximum point.
The exponent is approximated by a few terms of the Taylor expansion, and then the approximate integral is evaluated explicitly.
The key step is to find the appropriate contour. This is obtained first by evaluating the critical point of the exponent and then finding the contour of the constant-phase passing through this critical point.
For the case at hand, there is a difference to the standard method of steepest-descent: The critical point vanished to the second order instead of the first order. This leads to Airy-type functions instead of Gaussian functions in the end.
The second order vanishing  of the critical point is indeed how the centering in a limit theorem is determined.
This is a typical phenomenon in the application of the method of steepest-descent in the theory of random matrices (see e.g. \cite{Baik-Ben_Arous-Peche05} \cite{El_Karoui07}).
In the proofs that follows, we make a special effort to highlight this key step and the associated issue of the determination of the centering in our Theorem so that the reader may understand the origin of the phase transitions.

%Since the asymptotic analysis has many similarities with the previous works, we do not present all the technical details here. %Instead we give a sketch of the proof and focus on explaining
% how we can find the appropriate centering and scaling of the theorem and how the phase transitions occur.
%how the centering and scaling of the theorem arise and how the limiting distribution is obtained.

\subsection{Scaling and conjugation of determinant}

Before we take the limit, we first recall  two basic facts about Fredholm determinants \cite{Simon05}. The first  is that the determinant is invariant under scalings. Namely, let $A$ be the trace class operator acting on the space $L^2((t, \infty))$ with kernel $A(a, b)$ and let $B$ the scaled operator defined by the  kernel
$B(a,b)= rA(c+ra, c+rb)$  which acts on $L^2((c+rt, \infty))$.
Then
$\det(1-A)= \det(1-B)$.

The second is that the determinant is invariant under conjugations by multiplicative operators.
%Let $A$ be the operator with kernel $A(a,b)$ acting the space $L^2((t,\infty))$.
Let $f(a)$ be a non-vanishing function on $(t, \infty)$.
Let $C(a,b)=f(a) A(a,b)\frac1{f(b)}$.
Suppose that the operator $C$ on $L^2((x,\infty))$ defined by the kernel $C(a,b)$ is a bounded trace-class operator.
Then $\det(1-A)=\det(1-C)$.

These properties follow from the definition~\eqref{eq:Frdetdef}. The first property is obtained from a simple change of variables, and the second property is a consequence of a property of the determinants of finite matrices.

\subsection{Critical points}\label{sec:cr}

Fix $s\in\R$.
Recall that we take the limit $m,n\to \infty$ such that $m/n\to \gamma$ for some $\gamma \in (0,\infty)$.
For each case of Theorem~\ref{th:main2}, we will show that for some constants $x=x(\gamma)$ and $p\in (0,1)$, {the probability}
$\Prob\{ L(m,n)\le xm+sm^{p}\}$ converges to a function in $s$.
In the analysis below, we will determine $x$ and $p$.

From~\eqref{eq:PKxL}, $\Prob\{ L(m,n)\le xm+sm^{p}\} = \det(1-K_{xm+sm^{p}})$.
%Here we set $x=x_0+sm^{-\gamma}$.
%$\det(1-K_{(x_0+sm^{-\gamma})m})$.
The operator acts on the Hilbert space $L^2((xm+sm^{p}, \infty))$ which varies in $m$.
%We first change this to a space which does not depend on $m$.
From the scale invariance, we see that % discussed in the previous section, we see that
$\det(1-K_{xm+sm^{\gamma}})= \det(1-\KK_x)$ where $\KK_x$ is defined by the scaled kernel
\begin{equation}\label{eq:K1-4}
\begin{split}
	\KK_{x}(\xi, \eta)
	=  m^{p} K \big( xm+sm^{p} +m^{p} \xi  ,
	 xm+sm^{p} +m^{p} \eta \big).
\end{split}
\end{equation}
Note that now the Hilbert space $L^2((0,\infty))$ for the operator $\KK_x$ does not depend on $m$.
We have
\begin{equation}\label{eq:K1-5}
\begin{split}
	\mathbb{P} \bigg\{ \frac{L(m,n)-xm}{m^{p}} \le s \bigg\} = \det(1-\KK_x).
\end{split}
\end{equation}
By~\eqref{eq:K1-1}, the kernel~\eqref{eq:K1-4} equals %  can be written as
\begin{equation}\label{eq:K1-2}
\begin{split}
	\KK_x(\xi, \eta)
	& = \frac{-m^{p}}{(2\pi)^2} \int_{\Sigma_{ \{\alpha, 0\} } } \int_{\Sigma_{\{\mu_{(i)}\}_{i=1}^m} }
		 e^{ m (\FF_m(z;x)-\FF_m(w;x))}
	\frac{(z-\alpha)w e^{m^{p}((s+\eta) w-(s+\xi) z)}}{(w-\alpha)z(w-z)}   dz dw
\end{split}
\end{equation}
where
\begin{equation}\label{eq:K1-3}
\begin{split}
	\FF_m(z;x)
	&:= -\frac1{m} \sum_{i=1}^m \log (\mu_{(i)}-z)
	+ \frac{n}{m} \log z  - x z. %\\
%	&= - \int \frac{dH_m(y)}{y-z} + \frac{n}{m} \log z  -x z.
\end{split}
\end{equation}
%Note that each term of $\FF_m$ are of order $O(1)$. We do not include terms of smaller order here.
%Keep in mind that $x$ and $\gamma\in (0,1)$ are to be determined.

%\bigskip

%We now discuss how the scaling of $L(m,n)$ is determined.
%We evaluate the above double integral asymptotically using  the method of steepest-descent.

Now,
\begin{equation}\label{eq:K1-6}
\begin{split}
	\frac{d}{dz} \FF_m(z;x)=\FF_m'(z;x) &= \frac1{m} \sum_{i=1}^m \frac1{\mu_{(i)}-z} + \frac{n/m}{z} - x
	%=  \int \frac{dH_m(y)}{y-z} + \frac{n/m}{z}  -x
	= \ell_m(z)- x , \\
	\frac{d^2}{dz^2} \FF_m(z;x)=\FF_m''(z;x) & = \frac1{m} \sum_{i=1}^m \frac1{(\mu_{(i)}-z)^2}- \frac{n/m}{z^2}
	%=  \int \frac{dH_m(y)}{(y-z)^2} - \frac{n/m}{z^2}
	=  \ell'_m(z)
\end{split}
\end{equation}
where $\ell_m(z) := \frac1{m} l_m(z)$ with $l_m(z)$ defined in~\eqref{eq:lz}.
%\begin{equation}\label{eq:K1-6.5}
%\begin{split}
%	 \ell_m(z) := \frac1{m} \sum_{i=1}^m \frac1{\mu_{(i)}-z} + \frac{n/m}{z}= \frac1{m} l_m(z).
%\end{split}
%\end{equation}
Some examples of the graphs of $\ell_m(z)$ are in Figure~\ref{fig:functionlfirst} in Section~\ref{sec:main results}.
{It is easy to check that $\FF_m''$ is a strictly  increasing function in the interval $z\in (0, \mu_{(m)})$ from $-\infty$ to $+\infty$, 
and hence $\FF_m'$ has a unique minimizer in the interval.}
From the definition~\eqref{eq:solvecrit} of $\lambda_m$, we see that
$\crit_m= argmin_{z\in (0, \mu_{(m)})} \FF'_m(z; x)$.
Note that $\lambda_m$ is  independent of $x$.

The critical points of $\FF_m$ (i.e. the roots of $\FF'_m$) play a key role. % in the method of steepest-descent.
We note that since $\FF'_m$ is convex in the interval $(0, \mu_{(m)})$, %in that interval,
the number of real roots of $\FF'_m$ in $(0, \mu_{(m)})$ is $2$, $1$, and $0$
if $x- \ell_m(\crit_m)>0$, $x- \ell_m(\crit_m)=0$, and $x- \ell_m(\crit_m)<0$, respectively.

%If there are no real roots of $\FF'_m$
%in $(0, \mu_{(m)})$, then $\FF_m'$ has complex roots.

\subsection{Case (a)}

We are yet to determine $x$ and $p$.
Here is a heuristic argument how we determine $x$.
Note  the exponential term $e^{m(\FF_m(z;x)-\FF_m(w;x))}$ in the double integral~\eqref{eq:K1-2}.
If we can deform the contours to the path of steepest-descent for $\FF_m(z;x)$ and the path of steepest-descent for $-\FF_m(w;x)$, and use the method of steepest-descent, the leading contribution to the double integral becomes $e^{m(\FF_m(z_c;x)-\FF_m(w_c;x))}$ where
$z_c$ and $w_c$ are the critical points.
(It can be shown that since we need to deform the original contours to the new contours analytically, the critical points $z_c$ and $w_c$ should be in the strip $0<Re(z)<\mu_{(m)})$.)
Note that the path of steepest-descent for $-\FF_m(w;x)$ is the path of steepest-ascent for $\FF_m(w;x)$.
Hence $z_c$ and $w_c$ are both critical points of the same function $\FF_m(z;x)$.
Now unless $\FF_m(z_c;x)=\FF_m(w_c;x)$, the leading term is not of order $O(1)$.
This suggests that we should have $z_c=w_c$,
which is attained if $\FF_m$ has a unique critical point in $(0, \mu_{(m)})$. From the discussion in the last paragraph of the previous section, this happens if we take $x=\ell_m(\lambda_m)$.

We now set
\begin{equation}\label{eq:xdefa}
	x=\ell_m(\crit_m)
\end{equation}
and show that the application of the method of steepest-descent to~\eqref{eq:K1-4} indeed yields the desired asymptotic result.
In this case, from the discussion at the end of the last subsection, there is only one real root of $\FF'_m$ in $(0, \mu_{(m)})$, given by $z=\crit_m$, and hence since $\FF'_m$ is convex in the same interval, this implies that $\FF''_m(\crit_m; x)=0=\FF_m'(\crit_m;x)$.
%$\FF''_m(\crit_m; x)=0$.
%As $\crit_m$ is the unique critical point of $\FF_m(z;x)$ in
%the interval $(0, \mu_{(m)})$,
%$\FF_m'(\crit_m;x)=\FF''_m(\crit_m; x)=0$.
Via the same approach as in (\ref{eq:K1-6}), it is straightforward to check that $\FF_m'''(z;x)>0$ in $z\in (0, \mu_{(m)})$, and hence especially
$\FF_m'''(\crit_m;x)>0$.

Let $\Gamma_1$ be the path of the steepest-decent of $\FF_m(z;x)$ passing through the point $\crit_m$.
%This curve satisfies the constant-phase condition that  $Im (\FF_m(z;x)-\FF_m(\crit_m; x))=0$.
Since $\FF_m'(\crit_m;x)=\FF_m''(\crit_m;x)=0$ and $\FF_m'''(\crit_m)>0$,
we have $\FF_m(z;x)-\FF_m(\crit_m; x)= c (z-\crit_m)^3+O(|z-\crit_m|^4)$ for $c>0$ locally for $z$ near $\crit_m$.
Hence starting at $z=\crit_m$, there are three directions, given by the angles $\pi/3$, $\pi$, $-\pi/3$ about the positive real line,
to which the real part of $\FF_m(z;x)$ decreases most rapidly, i.e. the direction of the steepest-descent.
We use the path which goes off $\crit_m$ at the angles $\pi/3$ and $-\pi/3$ in the complex plane.
Since the paths of steepest-descent and steepest-ascent are given by the constant-phase condition
(equivalently, given by the integral curves of the vector field $\overline{\FF_m'(z)}$), the  full curve satisfies the equation $Im (\FF_m(z;x)-\FF_m(\crit_m; x))=0$.
See the solid curve in Figure~\ref{fig:steep1} for an example of the general shape of the path of steepest-descent.
The path of steepest-ascent can also be obtained in a similar way and its general shape is indicated by the dashed curve in Figure~\ref{fig:steep1}.
This curve goes off $\crit_m$ at the angles $2\pi/3$ and $-2\pi/3$.

\begin{figure}[htbp]
\centering
\includegraphics[scale=0.6]{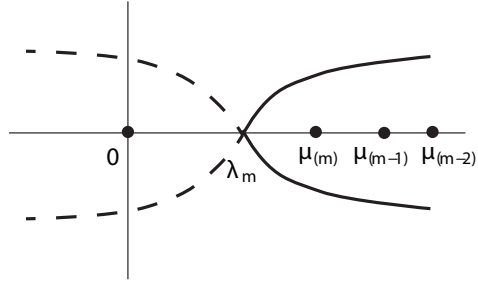}
\caption{General shape of the path of steepest-descent, $\Gamma_1$, (solid) and the path of steepest-ascent , $\Gamma_2$, (dashed) of
$\FF_m$ passing through the double critical point $z=\lambda_m$}
\label{fig:steep1}
\end{figure}

We deform the contour $\Sigma_{\{\mu_{(i)}\}_{i=1}^m}$ for $z$ to $\Gamma_1$ and the contour $\Sigma_{ \{\alpha, 0\} }$ for $w$ to $\Gamma_2$. (We orient the new contours consistent with the original contours.)
During this deformation, we need to be careful of the poles of the integrand.
The poles are $z=\mu_{(i)}$, $w= \alpha, 0$, and $z=w$.
Since we assume that (as we are in the case (a))
\begin{equation}\label{eq:assm1}
	\alpha<\liminf_{m\to\infty} \crit_m, \qquad \liminf_{m\to \infty} (\mu_{(m)}-\crit_m)>0,
\end{equation}
we see that we can  deform the contours to $\Gamma_1$ and $\Gamma_2$ without passing through
the poles $z=\mu_{(i)}$ and $w=\alpha, 0$.
About the pole $z=w$, even though the new contours meet at $z=w=\crit_m$,
we can modify the contours locally near the critical point as follows.
It can be shown that if we take the contours $\Gamma_1$ and $\Gamma_2$ be $\crit_m \pm O(m^{-1/3})$, respectively, near the critical point,
the pole $z=w$ does not contribute but the method of  steepest-descent still applies (see e.g. \cite{Baik-Ben_Arous-Peche05})

Since the critical point $z=\lambda_m$ is away from the poles due to~\eqref{eq:assm1},
the main contribution to the double integral comes from a small neighborhood of the critical point.
By localizing the integral near the critical point and expanding the exponent in a Taylor series,
we find the leading asymptotic term of $\KK_x(\xi, \eta)$:
%Then the kernel becomes
\begin{equation}\label{eq:assm2}
	\KK_x(\xi, \eta)
	\approx \frac{-m^{p}}{(2\pi)^2} \int \int
	e^{\frac{m}{3!}\FF_m'''(\crit_m;x)((z-\crit_m)^3-(w-\crit_m)^3)}
	\frac{e^{m^{p}((s+\eta) w-(s+\xi) z)} }{w-z}  dz dw,
\end{equation}
{
where the expression $f\approx g$ means that $\frac{f}{g}\to 1$ as $m\to\infty$ throughout the paper.}
Here we used the fact that $\frac{(z-\alpha)w}{(w-\alpha)z}=1$ at $z=w=\lambda_m$.
%\begin{equation}\label{eq:assm2}
%	\KK_x(\xi, \eta)	\approx \frac{-m^{p}}{(2\pi)^2} \int \int	e^{\frac{m}{3!}\FF_m'''(\crit_m;x)((z-\crit_m)^3-(w-\crit_m)^3)} \frac{(z-\alpha)w e^{m^{p}((s+\eta) w-(s+\xi) z)} }{(w-\alpha)z(w-z)}  dz dw.
%\end{equation}
Changing the variables %$z\mapsto u$, $w\mapsto v$by
$u:=c_m(z-\crit_m)$ and $v:=c_m(w-\crit_m)$ where $c_m:= (\frac{m}2\FF_m'''(\crit_m;x))^{1/3}=O(m^{1/3})$,
the above becomes
\begin{equation}\label{eq:assm3}
\begin{split}
	\KK_{x}(\xi, \eta)
	\approx
	&e^{-m^{p}\lambda_m \xi}
	\bigg[-  \frac{\frac{m^{p}}{c_m}}{(2\pi)^2} \int \int
	e^{\frac1{3}u^3-\frac13 v^3}
	\frac{e^{\frac{m^{p}}{c_m}((s+\eta) v-(s+\xi) u)} }{u-v} du dv \bigg]
	e^{m^{p}\lambda_m \eta}. %\\
%	=: &e^{-m^{p}\lambda_m \xi} \AAA_s(\xi, \eta) e^{m^{p} \lambda_m \eta}.
\end{split}
\end{equation}
{Here the contour for $u$ is from $e^{\pi i/3}\infty$ to $e^{-\pi i/3} \infty$, and the contour for $v$
is from $e^{2\pi i/3}\infty$ to $e^{4\pi i/3}\infty$ such that the first is to the right of the second: see Figure~\ref{fig:parabola}.}
\begin{figure}[htbp]
\centering
\includegraphics[scale=0.2]{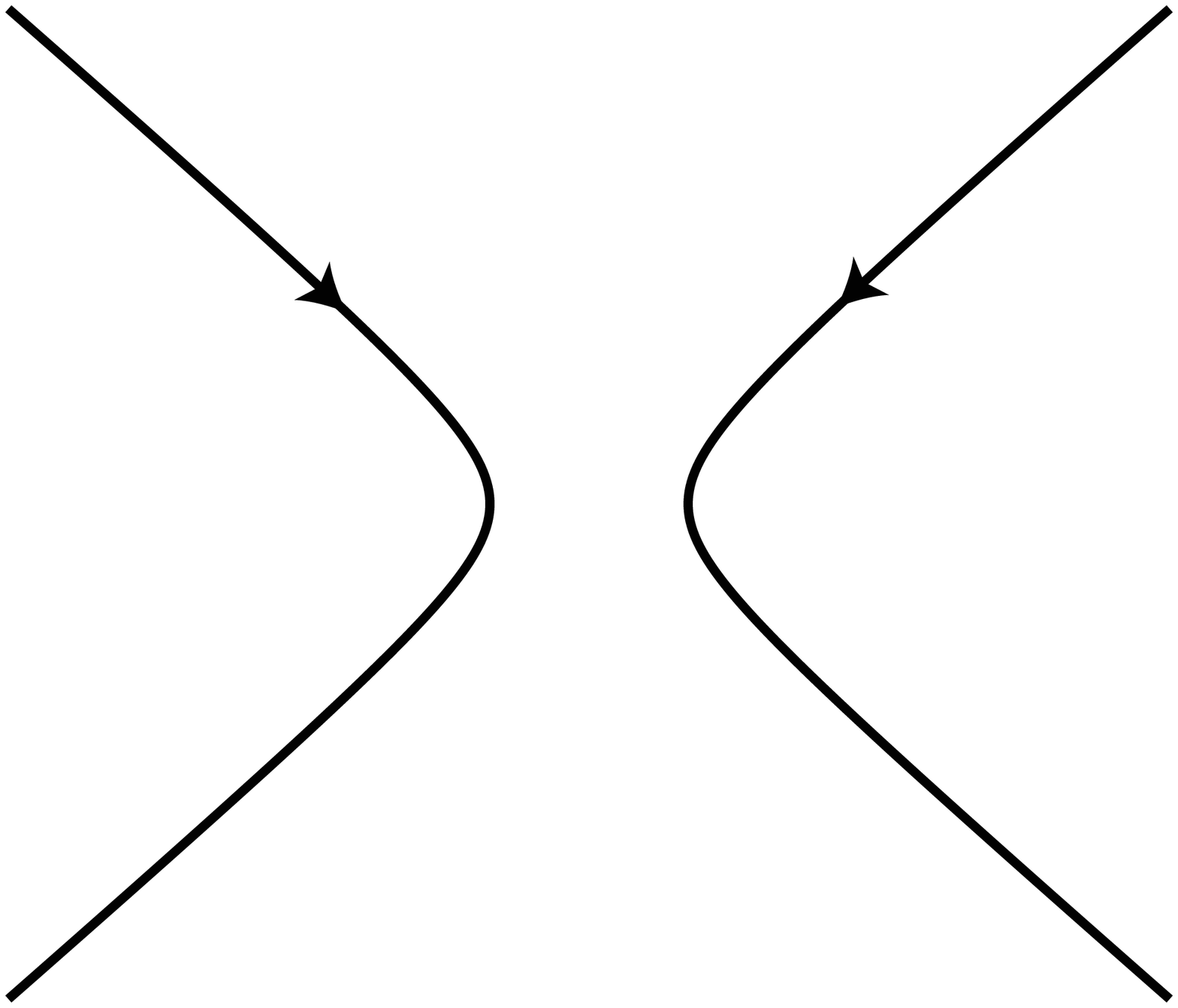}
\caption{Contours for~\eqref{eq:assm3}}
\label{fig:parabola}
\end{figure}
Since $c_m= O(m^{1/3})$, the term in the bracket is $O(1)$ if we take
\begin{equation}
	p= \frac13.
\end{equation}
This is how the parameter $p$ is determined.

Now, define the Airy kernel
\begin{equation}\label{eq:assm4.1}
\begin{split}
	\AAA_s(\xi, \eta)
	:= -  \frac{1}{(2\pi)^2} \int \int
	e^{\frac1{3}u^3-\frac13 v^3}
	\frac{e^{(s+\eta) v-(s+\xi) u} }{u-v} du dv
\end{split}
\end{equation}
and the Airy operator $\AAA_s$ on $L^2((0, \infty))$ defined by the above kernel.
Hence we have obtained
\begin{equation}\label{eq:assm4.1}
\begin{split}
	\KK_{x}(\xi, \eta)
	\approx
	&e^{-m^{1/3}\lambda_m \xi} \bigg[ \frac{m^{1/3}}{c_m} \AAA_{\frac{m^{1/3}s}{c_m}}\big(\frac{m^{1/3}}{c_m}\xi, \frac{m^{1/3}}{c_m}\eta \big) \bigg]
	e^{m^{1/3}\lambda_m \eta}.
\end{split}
\end{equation}
{By mimicking the arguments in Section 3 of \cite{Baik-Ben_Arous-Peche05} (see also \cite{El_Karoui07}), we can show that the difference between the left-hand side and the right-hand side, and its derivative, tends to zero uniformly.} This allows us to establish that the operator $\KK_{x}$ converges to the operator with the above kernel in trace norm. Since these arguments are `standard' and rather tedious, while providing no additional insight, they provide no archival value to this paper and so we do not copy them here. Having established convergence of $\KK_{x}$, we obtain, using the invariance of the Fredholm determinants under scaling and conjugation,
\begin{equation}\label{eq:assm4-1}
	\det(1-\KK_{x})
		\approx \det(1-\AAA_{\frac{m^{1/3}s}{c_m}}). %= \mathbb{P}(TW_2\le s) .
\end{equation}
From~\eqref{eq:K1-5}, the left-hand side equals $\mathbb{P} \big\{ L(m,n)\le m  \ell_m(\crit_m)+ m^{1/3}s \big\}$.
%that (see~\eqref{eq:K1-5})
%\begin{equation}\label{eq:assm4-1}
%	\mathbb{P} \bigg\{ \frac{L(m,n)-m  \ell_m(\crit_m)}{m^{1/3}} \le s \bigg\}
%	\approx \det(1-\AAA_{\frac{m^{1/3}s}{c_m}}). %= \mathbb{P}(TW_2\le s) .
%\end{equation}
Therefore, after changing $s$ to $\frac{c_m}{m^{1/3}}s$, we obtain
\begin{equation}\label{eq:assm4}
	\mathbb{P} \bigg\{ \frac{L(m,n)-m  \ell_m(\crit_m)}{(\frac{m}2\FF_m'''(\crit_m;x))^{1/3} }
	\le s \bigg\}
	\approx \det(1-A_s) = \mathbb{P}(\TWC\le s).
\end{equation}
The last equality is one of the definitions of the Tracy-Widom distribution \cite{Tracy-Widom94}.
Since $\ell_m(z)=\frac{1}{m} l_m$, we conclude that
\begin{equation}\label{eq:assm5}
	\frac{L(m,n)-l_m(\crit_m)}{(\frac{1}2 l_m''(\crit_m;x))^{1/3}}
	 \overset{\mathcal{D}}{\longrightarrow}  \TWC,
\end{equation}
and hence Theorem~\ref{th:main2} (a) is proved.

\subsection{Case (b)}

In this case the assumption is that
\begin{equation}\label{eq:assm7}
	\limsup_m \crit_m < \alpha.
\end{equation}
We are to determine $x$ and the exponent $p\in (0,1)$ of~\eqref{eq:K1-5}.

If we take $x$ as~\eqref{eq:xdefa} and take the same new contours as in the previous section,
then due to the condition~\eqref{eq:assm7}, the deformation from the original contour  to the new contour for $w$-variable passes through the pole $w=\alpha$.
In this case the leading contribution to the $w$-integral does not come from the critical point $\lambda_m$, but comes from the pole $\alpha$.
Then since $\FF_m(\lambda_m,x)\neq \FF_m(\alpha, x)$, the leading contribution to the double integral is
not $O(1)$.
This means that the choice~\eqref{eq:xdefa} is not suitable for case (b).

Unlike the case (a) where we defined $x$ so that there is a unique critical point of $\FF_m$ in $(0, \mu_{(m)})$, we now assume that we have $x>\ell_m(\lambda_m)$ so that there are
two real critical values $z_c^{-}<z_c^+$ in $(0, \mu_{(m)})$.
Note that $z_c^{\pm}$ depend on $x$, which is to be determined. %=z_c^-(x)$ and $z_c^+=z_c^+(x)$.
As $\FF_m'(z;x)= \ell_m(z)-x$ and $\ell_m(z)$ is convex in $(0, \mu_{(m)})$ (see the last three paragraphs of Section~\ref{sec:cr}),
we see that $z_c^-\in (0, \lambda_m)$ and $z_c^+\in (\lambda_m, \mu_{(m)})$.
(Recall that the definition of $\lambda_m$ does not involve $x$.)
It is easy to check that the path of steepest-descent of $\FF_m$ passing through the critical point $z_c^+$
(the future contour for the $z$-integral) is locally a vertical line near $z_c^+$ and the path of steepest-ascent of $\FF_m$ (the future contour for the $w$-integral)  passing through  the critical point $z_c^-$
is locally a vertical linear near $z_c^-$.
The general shape of these paths are shown in Figure~\ref{fig:steep12}.
If we deform the original contours to these contours, the deformation of the $w$-integral, whose original contour is $\Sigma_{\{0, \alpha\}}$, passes the pole $\alpha$ due to the
condition~\eqref{eq:assm7}.
\begin{figure}[htbp]
\centering
\includegraphics[scale=0.6]{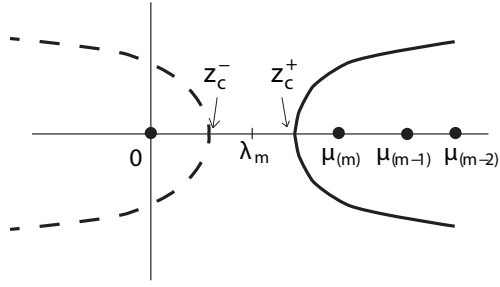}
\caption{General shape of the path of steepest-descent of $\FF_m$ passing through the point $z_c^+$(solid) and
the general shape of the path of steepest-ascent passing through the point $z_c^-$ (dashed)}
\label{fig:steep12}
\end{figure}

With this in mind, before applying the method of steepest-descent,
we first deform the contours of the double integral~\eqref{eq:K1-5} so that
so that $w=\alpha$ is outside of the contour for $w$.
By evaluating the residue at $w=\alpha$, we find {
\begin{equation}
	\KK_x(\xi, \eta)=  I+J
\end{equation}}
where
\begin{equation}\label{eq:assm8-010}
\begin{split}
	I =  \bigg[ \frac{-m^{p}\alpha}{2\pi i} \int_{\Sigma_{\{\mu_{(i)}\}_{i=1}^m} }
		 e^{ m \FF_m(z;x)}
	\frac{e^{-m^{p} (s+\xi) z} }{z}  dz \bigg]
	e^{-m\FF_m(\alpha; x)+m^{p}(s+\eta)\alpha}
\end{split}
\end{equation}
and
\begin{equation}\label{eq:assm8-020}
\begin{split}
	J= \frac{-m^{p}}{(2\pi)^2} \int_{\Sigma_{ \{0\} } } \int_{\Sigma_{\{\mu_{(i)}\}_{i=1}^m} }
		 e^{ m (\FF_m(z;x)-\FF_m(w;x))}
	\frac{(z-\alpha)w e^{m^{p}((s+\eta) w-(s+\xi) z)}}{(w-\alpha)z(w-z)}   dz dw.
\end{split}
\end{equation}
Note that the contour for $w$ in $J$ contains $0$ inside, but $\alpha$ outside.
Now we deform the contours to the paths of steepest-descent/steepest-ascent described above.
Then the leading term of $J$
is $e^{m(\FF_m(z_c^+(x);x)-\FF_m(z_c^-(x);x))}$.
But since $\FF_m'$ is convex in $(0, \mu_{(m)})$, we see that $\FF_m(z_c^+(x);x)-\FF_m(z_c^-(x);x)<0$.
Thus the double integral is exponentially small for any choice of $x>\ell_m(\lambda_m)$. 
{Hence $J=O(e^{-cm})$ for some constant $c>0$.}

On the other hand, the integral in $I$
has the leading term $e^{m\FF_m(z_c^+(x);x)}$.
If this term is same as $e^{-m\FF_m(\alpha; x)}$, then $I=O(1)$.
This is achieved if $z_c^+(x)=\alpha$, which means that $\alpha$ is one of the critical points of $\FF_m(z;x)$, i.e. $\ell_m(\alpha)=x$. We choose $x$ to satisfy
\begin{equation}\label{eq:conditionforxb}
	x= \ell_m(\alpha).
\end{equation}
With the above choice of $x$, the method of steepest-descent yields that
\begin{equation}\label{eq:assm9}
\begin{split}
	I \approx
	\bigg[\frac{-m^{p}}{2\pi i} \int e^{ \frac12 m \FF_m''(\alpha;x)(z-\alpha)^2}
	e^{-m^{p} (s+\xi) z} dz \bigg]
	e^{m^{p}(s+\eta)\alpha}
\end{split}
\end{equation}
where the integral is localized near the critical point $z_c^+=\alpha$.
Changing the variables $u:= \sqrt{m\FF_m''(\alpha;x)}(z-\alpha)$, we see that
the two exponents in the integral are balanced if we take
\begin{equation}
	p=\frac12.
\end{equation}
With this choice of $p$, we find that %the first term of~\eqref{eq:assm8}
\begin{equation}\label{eq:assm10}
\begin{split}
	e^{m^{1/2} \alpha\xi}   Ie^{-m^{1/2}\alpha\eta}
	\approx \frac{-1}{2\pi i\beta} \int  e^{ \frac12 u^2}
	e^{-\frac1{\beta} (s+\xi) u}  du
	=  \frac{1}{\sqrt{2\pi}\beta}
	e^{-\frac1{2\beta^2}(s+\xi)^2} ,
\end{split}
\end{equation}
and hence
\begin{equation}\label{eq:assm10}
\begin{split}
	I\approx e^{-m^{1/2} \alpha\xi}  \bigg[ \frac{-1}{2\pi i\beta} \int
		 e^{ \frac12 u^2}
	e^{-\frac1{\beta} (s+\xi) u}  du \bigg] e^{m^{1/2}\alpha\eta}
	= e^{-m^{1/2} \alpha\xi}  \bigg[ \frac{1}{\sqrt{2\pi}\beta}
	e^{-\frac1{2\beta^2}(s+\xi)^2} \bigg] e^{m^{1/2}\alpha\eta}.
\end{split}
\end{equation}
{Thus, we find that $\KK_x\approx I$.} Using the invariance of determinants under conjugations, we find that
\begin{equation}\label{eq:assm12}
\begin{split}
	\mathbb{P} \bigg\{ \frac{L(m,n)-m  \ell_m(\alpha)}{m^{1/2}} \le s \bigg\}
	=\det(1-\KK_x)
	\approx \det(1-\GGG)
\end{split}
\end{equation}
where $\GGG$ is an operator on $L^2((0,\infty))$ defined by the kernel
$\GGG(\xi, \eta)=  g(\xi)= \frac{1}{\sqrt{2\pi}\beta} e^{-\frac1{2\beta^2}(s+\xi)^2}$.
Since $\GGG(\xi, \eta)$ does not depend on $\eta$, it is easy to check that the only eigenfunction of $\GGG$ is $g(\eta)$ with the eigenvalue $\int_0^\infty g(\eta)d\eta$.
Hence the determinant
\begin{equation}\label{eq:assm12-1}
\begin{split}
	\det(1-\GGG)= 1- \int_0^\infty g(\eta) d\eta =
	 \frac1{\sqrt{2\pi}\beta} \int_{-\infty}^s e^{-\frac1{2\beta^2} \eta^2} d\eta,
\end{split}
\end{equation}
which is the cdf of a normal distribution.
Therefore, changing $s\mapsto \beta s$, we obtain
\begin{equation}\label{eq:assm14}
\begin{split}
	\frac{L(m,n) - l_m(\alpha)}{\sqrt{l_m'(\alpha)}}
	\approx  \mathcal{N}(0,1)
\end{split}
\end{equation}
and Theorem~\ref{th:main2} (b) is proved.

\subsection{Case (c)}

The assumptions are
\begin{equation}\label{eq:assm16}
	\limsup_{m\to \infty} ( \mu_{(m)} -\crit_m)= 0,
	\qquad \liminf_{m\to \infty} (\mu_{(m-1)}-\mu_{(m)})>0,
\end{equation}
%\begin{equation}\label{eq:assm17}
%	\liminf_{m\to \infty} (\mu_{(m-1)}-\mu_{(m)})>0,
%\end{equation}
and
\begin{equation}\label{eq:assm16.5}
	\displaystyle{\liminf_{m\to\infty}}  (\lambda^{(1)}_m-\mu_{(m)}) >0.
\end{equation}
%For simplicity of presentation, we first assume that
%\begin{equation}\label{eq:assm17}
%	\liminf_{m\to \infty} (\mu_{(m-1)}-\mu_{(m)})>0.
%\end{equation}
The case when $\mu_{(m)}=\cdots=\mu_{(m-r+1)}$  in Remark~\ref{rmk:fluctuationc} will be discussed at the end.

Due to the condition~\eqref{eq:assm16}, one of the real critical points (assuming that $x>\ell_m(\lambda_m)$ again) becomes close to the pole $\mu_{(m)}$ and hence
we cannot argue that the main contribution is localized near the critical point.
%we cannot use the same analysis before.
In this case, we first change the contour $\Sigma_{\{\mu_{(i)}\}_{i=1}^m}$ for the $z$-integral in~\eqref{eq:K1-2}
to $\Sigma_{\{\mu_{(i)}\}_{i=1}^{m-1}}$ which excludes the pole $\mu_{(m)}$.
(Recall that in the previous section, we changed the contour $\Sigma_{\{0, \alpha\}}$ for the $w$-integral to $\Sigma_{\{0\}}$.)
Evaluating the residue at $z=\mu_{(m)}$, we obtain
$\KK_x(\xi, \eta)=L+M$ where
\begin{equation}\label{eq:assm18-010}
\begin{split}
	L=
	 e^{m\FF_m^{(1)}(\mu_{(m)}; x)-m^{p}(s+\xi)\mu_{(m)}}
	\bigg[ \frac{m^{p}(\mu_{(m)}-\alpha)}{2\pi i \mu_{(m)}} \int_{\Sigma_{\{0, \alpha\}} }
		 e^{- m \FF_m^{(1)}(w;x)}
	\frac{w e^{m^{p} (s+\eta) w}}{w-\alpha}    dw \bigg]
\end{split}
\end{equation}
and
\begin{equation}\label{eq:assm18-020}
\begin{split}
	M= \frac{-m^{p}}{(2\pi)^2} \int_{\Sigma_{ \{0, \alpha\} } } \int_{\Sigma_{\{\mu_{(i)}\}_{i=1}^{m-1}}}
		 e^{ m (\FF_m^{(1)}(z;x)-\FF_m^{(1)}(w;x))}
	\frac{(\mu_{(m)}-w)(z-\alpha)w e^{m^{p}((s+\eta) w-(s+\xi) z)} }{(\mu_{(m)}-z)(w-\alpha)z(w-z)}   dz dw,
\end{split}
\end{equation}
where
\begin{equation}\label{eq:assm19}
\begin{split}
	\FF_m^{(1)}(z;x)
	&:= -\frac1{m} \sum_{i=1}^{m-1} \log (\mu_{(i)}-z)
	+ \frac{n}{m} \log z  - x z.
\end{split}
\end{equation}
We set (see~\eqref{eq:lmrde})
\begin{equation}\label{eq:assm20}
\begin{split}
	\ell_m^{(1)}(z):=  \frac1{m} \sum_{i=1}^{m-1} \frac1{\mu_{(i)}-z}
	+ \frac{n/m}{z} = \frac1{m} l^{(1)}(z).
\end{split}
\end{equation}
Then
\begin{equation}\label{eq:assm20-1}
	(\FF_m^{(1)})'(z;x)= \ell_m^{(1)}(z)-x.
\end{equation}

The function $\ell_m^{(1)}(z)$ is convex in $(0, \mu_{(m-1)})$
and $\crit_m^{(1)}$ is defined to be the unique solution of  $(\ell_m^{(1)})'(z)=0$ in $(0, \mu_{(m-1)})$.
As $\ell_m^{(1)}(z)=\ell_m(z)-\frac1{m(\mu_{(m)}-z)}$, we have
$(\ell_m^{(1)})'(\lambda_m)<0$ and hence $\lambda_m^{(1)}> \lambda_m$.
Thus from the condition~\eqref{eq:assm16}, we find that
$\displaystyle{\liminf_{n\to\infty}} (\crit_m^{(1)}-\mu_{(m)})\ge 0$.
The condition~\eqref{eq:assm16.5} implies that indeed here the inequality holds strictly.

Now the rest of the analysis is similar to the previous section with the role of $\alpha$ is now played by $\mu_{(m)}$.
The proper choice of $x$ is now (cf. \eqref{eq:conditionforxb})
\begin{equation}\label{eq:assm22}
	x=\ell_m^{(1)}(\mu_{(m)}).
\end{equation}
Then from~\eqref{eq:assm20-1} and~\eqref{eq:assm16.5}, we find that there are two
real roots of $(\FF_m^{(1)})'(z;x)=0$ in $z\in (0, \mu_{(m-1)})$.
The smaller root is $\mu_{(m)}$, and we denote the other root by  $z_0\in (\mu_{(m)}, \mu_{(m-1)})$.
The path of steepest-descent of $\FF_m^{(1)}$ passing through $z_0$ and the path of steepest-ascent passing through $\mu_{(m)}$ are of shape in Figure~\ref{fig:steep3}.
%We evaluate the integrals in~\eqref{eq:assm18} using these paths.
\begin{figure}[htbp]
\centering
\includegraphics[scale=0.6]{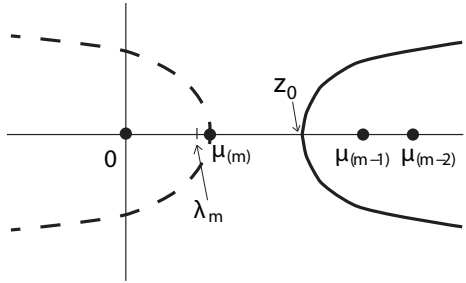}
\caption{General shape of the path of steepest-descent of $\FF_m^{(1)}$ passing through the point $z_0$(solid) and
the general shape of the path of steepest-ascent passing through the point $\mu_{(m)}$ (dashed)}
\label{fig:steep3}
\end{figure}

Since $(\FF_m^{(1)})''(z;x)\neq 0$ at $z=\mu_{(m)}$ and at $z=z_0$, we are lead to choose $p=\frac12$.
%\begin{equation}
%	p=\frac12.
%\end{equation}
With this choice, the evaluation of $L$ is similar to that of $I$ in the previous section and we obtain
\begin{equation}\label{eq:assm23}
\begin{split}
	L\approx e^{-m^{1/2} \mu_{(m)}\xi}  \bigg[ \frac{1}{\sqrt{2\pi}\beta}
	e^{-\frac1{2\beta^2}(s+\eta)^2} \bigg] e^{m^{1/2}\mu_{(m)}\eta},
	\quad \beta:= \sqrt{-(\ell_m^{(1)})'(\mu_{(m)})}.
\end{split}
\end{equation}

On the other hand, the method of steepest-descent implies that
%second term in~\eqref{eq:assm18} is of order
$M$ is of order $O(e^{m(\FF_m^{(1)}(z_0;x)-\FF_m^{(1)}(\mu_{(m)};x))})$.
Since $\ell_m^{(1)}$ is convex in $z\in (0, \mu_{(m-1)})$,
the exponent is negative and hence the second term is exponentially small.
This argument works without any problem if $\liminf_m(\mu_{(m-1)}-z_0)>0$.
(Note that the double integral has a pole at $z=\mu_{(m-1)}$.)
Even if $\limsup_m(\mu_{(m-1)}-z_0)=0$,
we can still show that the second term is exponentially small.
This is because  $\FF_m^{(1)}(z;x)$ is decreasing as $z$ increases from $\mu_{(m)}$ to $z_0$
and hence it is possible to choose a $z$-contour which passes the real axis in-between $\mu_{(m)}$ and $z_0$.
We skip the details.

Therefore, we find that
\begin{equation}
\begin{split}
	\mathbb{P} \bigg\{ \frac{L(m,n)-m  \ell_m(\mu_{(m)})}{m^{1/2}} \le s \bigg\}
	=\det(1-\KK_x)
	\approx \det(1-\GGG)
\end{split}
\end{equation}
where $\GGG$ is an operator on $L^2((0,\infty))$ defined by the kernel
$\GGG(\xi, \eta)=  g(\xi)= \frac{1}{\sqrt{2\pi}\beta} e^{-\frac1{2\beta^2}(s+\xi)^2}$
as in~\eqref{eq:assm12} with the new $\beta$ defined in~\eqref{eq:assm23}.
From~\eqref{eq:assm12-1}, we find, after changing $s\mapsto \beta s$, that
%we obtain, with $l_m^{(1)}:=m\ell_m^{(1)}$,
\begin{equation}\label{eq:assm24}
\begin{split}
	\frac{L(m,n) - l_m^{(1)}(\mu_{(m)})}{\sqrt{-(l_m^{(1)})'(\mu_{(m)})}}
	\approx  \mathcal{N}(0,1)
\end{split}
\end{equation}
and Theorem~\ref{th:main2} (c) is proved.

\medskip

Finally we discuss Remark~\ref{rmk:fluctuationc}.
If $\mu_{(m)}= \cdots= \mu_{(m-r)}$ for some $r$ independent of $m$ and
$\displaystyle{\liminf_{m\to \infty}}(\mu_{(r+1)}-\mu_{(r)})>0$
as in Remark~\ref{rmk:fluctuationc}, the pole $z=\mu_{(m)}$ is not simple but is of order $r$.
This changes the formula of $L$ and $M$ in~\eqref{eq:assm18-010} and~\eqref{eq:assm18-020}. However the rest of the analysis is same except that in the end the limit of $L$ is of different form.
The limit is $\det(1-\GGG^{(r)})$ where the kernel of the operator $\GGG^{(r)}$ is a certain rank $r$ generalization of $\GGG(\xi, \eta)$.
This kernel appeared in \cite{Baik-Ben_Arous-Peche05} in which it was shown that $\det(1-\GGG^{(r)})$ is the distribution function of the largest eigenvalue of $r\times r$ matrix from the Gaussian unitary ensemble.

\section{Appendix. Basics of the method of steepest-descent}

The asymptotic analysis of the kernel $K$ in Section~\ref{sec:proof2} was done by using the method of steepest-descent. 
For the benefit of the unfamiliar readers, here we briefly discuss some basics of the method of steepest-descent. 

We first consider the Laplace method. The method of steepest-descent can be thought of as the Laplace method for complex functions.  
Suppose we are interested in the large $t$ asymptotics of the integral $\int_{\R} e^{t f(x)} dx$ where $f:\R\to \R$ is a smooth real function. 
Assume, furthermore, that $f$ has a unique maximum value obtained at $x=x_c$ 
(hence $f'(x_c)=0$), and $f''(x_c)\neq 0$. 
Then necessarily $f''(x_c)<0$, and for $x$ close to $x_c$, $f(x)= f(x_c)+\frac12 f''(x_c)(x-x_x)^2+O((x-x_c)^3)$. 
Since $f$ has a unique maximum, we can choose a small interval $I=[x_c-\epsilon, x_c+\epsilon]$ around $x_c$ so that the value of the function $f(x)$ for $x$ outside this interval is strictly less than the minimum of $f(x)$, $x\in I$. 
When $t\to \infty$, the value of $e^{tf(x)}$ for $x\in\R\setminus I$ is exponentially smaller than $e^{tf(x)}$, $x\in I$. 
From this, we can imagine that 
\begin{equation}\label{eq:A1}
	\int_{\R} e^{t f(x)} dx \approx \int_{I} e^{t f(x)} dx.
\end{equation}
This can be shown easily under some mild extra conditions on the behavior of $f(x)$ as $|x|\to\infty$. 
Now in the small interval $I=[x_c-\epsilon, x_c+\epsilon]$, we may approximate $f$ by its Taylor expansion $f(x_c)+\frac12 f''(x_c)(x-x_x)^2$, and hence it is plausible to expect that 
\begin{equation}\label{eq:A2}
	\int_{I} e^{t f(x)} dx \approx \int_{x_c-\epsilon}^{x_c+\epsilon} e^{tf(x_c)+\frac{t}2 f''(x_c)(x-x_x)^2}dx .
\end{equation}
Indeed this can be shown to be true for very general functions $f$. 
Finally, the right-hand side can be approximated by the integral over $\R$, 
\begin{equation}\label{eq:A3}
	\int_{x_c-\epsilon}^{x_c+\epsilon} e^{tf(x_c)+\frac{t}2 f''(x_c)(x-x_x)^2 }dx \approx  e^{tf(x_c)} \int_{-\infty}^{\infty} e^{\frac{t}2 f''(x_c)(x-x_x)^2 }dx.
\end{equation}
because the last integral over $\R\setminus I$ is $O(e^{-ct})$ for some $c>0$.
Combining~\eqref{eq:A1},~\eqref{eq:A2}, and~\eqref{eq:A3}, and evaluating the Gaussian integral (note that $f''(x_c)<0$), we obtain
\begin{equation}\label{eq:A4}
	\int_{\R} e^{tf(x)}dx \approx  e^{tf(x_c)} \sqrt{\frac{2\pi}{-tf''(x_c)}}.
\end{equation}

Note that the method above can be applied without change to the integrals $\int_{C} e^{tf(z)}dz$ for a general (smooth) contour $C$ in the complex plane as long as $f(z)$ is real-valued for $z\in C$. 

Now suppose that we would like to evaluate the large $t$ asymptotics of $\int_\Sigma e^{tf(z)}dz$ where $\Sigma$ is a contour in $\C$ and $f:\Sigma\to \C$ is an analytic complex function. 
Write $f(z)=u(z)+iv(z)$ where $u$ and $v$ are real functions. 
Suppose that we were able to find a contour $\Sigma'$ such that the imaginary part of $f$ is a constant on $\Sigma'$: there is $c\in \R$ such that $v(z)=c$ for all $z\in \Sigma'$. 
Then using the Cauchy's theorem, we can deform the contour $\Sigma$ to $\Sigma'$ and the integral becomes $e^{itc} \int_{\Sigma'} e^{tu(z)}dz$. 
This integral can be  evaluated asymptotically using the Laplace method if the function $u(z)$ has a unique maximum on $\Sigma'$. 
Thus we need to choose a contour $\Sigma'$ such that the imaginary part $v$ of $f$ is a constant and the real part $u$ of $f$ has a critical point.
Let us now find a condition for such a contour. 
Since $v$ is a constant on the contour, its derivative along the contour is zero. 
And hence at the critical point $z_c$ of $u$, we should have $f'(z_c)=0$ due to the Cauchy-Riemann equations. 
Thus the desired contour $\Sigma'$ should pass through the critical point of $f$ (in the complex sense). 
Hence we first find the critical point(s) $z_c$ of $f$ and then choose the contour $\Sigma'$ passing through $z_c$ given by $v(z)=constant=v(z_c)$. 
Since the level curves of $u$ and $v$ are orthogonal (due to the Cauchy-Riemann equation), $u(z)$ decays as $z$ travels away from $z_c$ along $\Sigma'$. 
(And this is the path of the steepest-descent for $u$, and hence the method is called the method of steepest-descent.) 
After we deform the original contour to the new contour, we apply the Laplace method. 

If the function has isolated singularities, then one may need to pass through the singular points whiling deforming the original contour to the curve of steepest-descent. 
In that case the contribution from the singular points may be larger than that from the curve of steepest-descent. 
This happens in case (b) and (c) in the previous section.

When we apply the Laplace method, it may happen that $f''$ also vanishes at the critical point. Then the Gaussian integral~\eqref{eq:A3} is changed to a different  integral. This happens in the previous section. 

\section{Appendix: The $\TWC$ distribution} \label{sec:tw2}
\begin{figure}[h]
\centering
\includegraphics[width=0.9\textwidth]{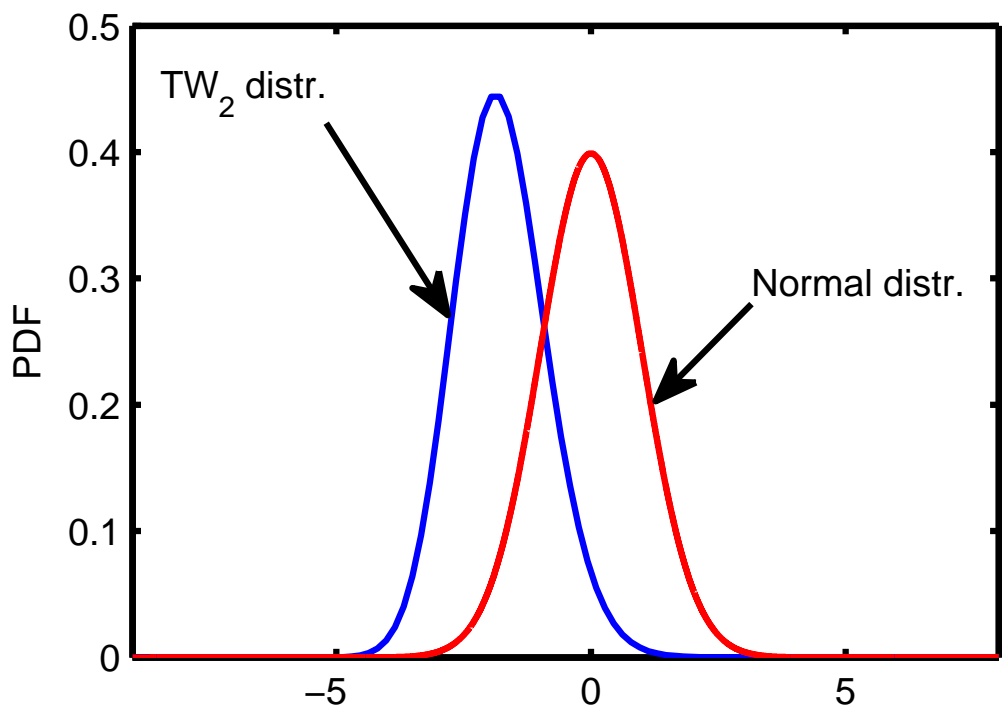}
\caption{The probability density functions of the $\TWC$ and $\mathcal{N}(0,1)$ distributions.}
\label{fig:twnormal}
\end{figure}

The $\TWC$ (or complex Tracy-Widom) distribution can be computed from the solution of the Painlev\'e II equation:
\begin{align} \label{eq:painleve}
q''=sq+2q^3
\end{align}
with the boundary condition
\begin{align} \label{eq:bndcond}
q(s)\sim \mathrm{Ai}(s),
   \qquad\textrm{as }s\rightarrow \infty.
\end{align}
The probability distribution $f_2(s)$, has pdf given by
\begin{equation} \label{eq:f2}
f_2(s)=\frac{d}{ds}F_2(s),
\end{equation}
where
\begin{equation} \label{eq:F2}
F_2(s)=\exp \left( -\int_s^\infty (x-s)q(x)^2\,dx \right).
\end{equation}
These distributions can be readily computed numerically.  See \cite{edelman2005random} for a simple solution and \cite{bornemann2010numerical,bornemann2010survey} for more accurate technique. The $\TWC$ distribution has a mean of $-1.771086807411$ and a variance of $0.8131947928329$. The $\TWC$ density is $O(e^{-\frac{4}{3} |x|^{3/2}})$ as $x\to +\infty$, and is $O(e^{-c|x|^{3}})$ as $x\to -\infty$ for any $c<\frac1{12}$ unlike the Gaussian \cite{Tracy-Widom94}.  Figure \ref{fig:twnormal} compares the `standard'' $\TWC$ distribution, computed using the methods described in \cite{bornemann2010numerical}, to the standard (zero mean, unit variance) normal distribution. Table  \ref{tab:tw quantiles} lists some quantiles of the $\TWC$ distribution that the reader might find useful in the context of hypothesis testing based on the distribution of $L(m,n)$.

See \cite{tracy2009distributions,Akemann-Baik-DiFrancesco11} for additional applications and context where the Tracy-Widom distributions arise.

\begin{table}[ht]
\centering
\begin{tabular}{|c|c|c|}
\hline
$\alpha$ &$1-\alpha$ & $\TWC^{-1}(1-\alpha)$ \\
\hline
0.990000  &0.010000  &-3.72444594640057\\
0.950000  &0.050000  &-3.19416673215810\\
0.900000  &0.100000    &-2.90135093847591\\
0.700000  &0.300000  &-2.26618203984916\\
0.500000  &0.500000  &-1.80491240893658\\
0.300000  &0.700000  &-1.32485955606020\\
0.100000  &0.900000     &-0.59685129711735\\
0.050000  &0.950000  &-0.23247446976400\\
0.010000  &0.990000  &0.47763604739084\\
0.001000  &0.999000  &1.31441948008634\\
0.000100  &0.999900  &2.03469175457082\\
0.000010  &0.999990  &2.68220732168978\\
0.000001 &0.999999  &3.27858828203370\\
\hline
\end{tabular}
\caption{The percentiles of the $\TWC$ distribution computed using software described in \cite{bornemann2010survey}.}
\label{tab:tw quantiles}
\end{table}

\section{Acknowledgements}
J.B.'s work was support in part by NSF grants DMS-1068646.
R.R.N's work was supported in part by NSF grant CCF-1116115. We thank the reviewers and the editors for their careful reading of our manuscript and for their feedback and many valuable suggestions.

%\bibliographystyle{abbrv}
%\bibliography{ref,more_ref}

\def\cydot{\leavevmode\raise.4ex\hbox{.}}

\end{document}